\pdfoutput=1
\documentclass[reqno]{amsart}

\usepackage{amsmath,amssymb,amsthm,verbatim, csquotes,mathrsfs,stmaryrd}
\usepackage{hyperref}
\usepackage{xcolor}
\usepackage{esint}
\usepackage{mathtools}
\usepackage{graphicx}
\usepackage{float}
\usepackage{caption}
\mathtoolsset{showonlyrefs}

\usepackage[top=3cm, bottom=4.0cm, right=2.5cm, left=2.5cm]{geometry}

\newtheorem{theorem}{Theorem}[section]
\newtheorem{lemma}[theorem]{Lemma}
\newtheorem{proposition}[theorem]{Proposition}
\newtheorem{definition}[theorem]{Definition}
\newtheorem{corollary}[theorem]{Corollary}
\newtheorem{remark}[theorem]{Remark}

\def\om{\omega}
\def\Om{\Omega}
\def\p{\partial}

\def\de{\delta}
\def\De{\Delta}

\def\S{{\Sigma}}
\def\<{\langle}
\def\>{\rangle}
\def\div{{\rm div}}
\def\na{\nabla}

\providecommand{\abs}[1]{\lvert#1\rvert}
\providecommand{\Abs}[1]{\left\lvert#1\right\rvert}

\providecommand{\norm}[1]{\lVert#1\rVert}

\newcommand{\mbN}{\mathbb{N}}

\newcommand{\mbR}{\mathbb{R}}
\newcommand{\mbS}{\mathbb{S}}

\newcommand{\mcA}{\mathcal{A}}
\newcommand{\mcB}{\mathcal{B}}

\newcommand{\mcH}{\mathcal{H}}

\newcommand{\mfc}{\mathbf{c}}

\newcommand{\mfC}{\mathbf{C}}

\newcommand{\rd}{{\rm d}}

\newcommand{\ra}{\rightarrow}

\newcommand{\eq}[1]{\begin{equation}\begin{alignedat}{2} #1 \end{alignedat}\end{equation}}

\numberwithin{equation} {section}

\begin{document}

	
\title[Anisotropic minimal graph]{A half-space Liouville theorem for anisotropic minimal graph with free boundary}
\date{\today}

\author[Wang]{Guofang Wang}
\address[G.W]{Mathematisches Institut\\
	Universit\"at Freiburg\\
	Ernst-Zermelo-Str.1\\
	79104\\
	\newline Freiburg\\ Germany}
\email{guofang.wang@math.uni-freiburg.de}

\author[Wei]{Wei Wei}
\address[W.W]{School of Mathematics\\ Nanjing University\\ 210093\\Nanjing\\ P.R. China}
\email{wei\_wei@nju.edu.cn}

\author[Xia]{Chao Xia}
\address[C.X]{School of Mathematical Sciences\\
	Xiamen University\\
	361005, Xiamen, P.R. China}
\email{chaoxia@xmu.edu.cn}

\author[Zhang]{Xuwen Zhang}
\address[X.Z]{Mathematisches Institut\\
	Universit\"at Freiburg\\
	Ernst-Zermelo-Str.1\\
	79104\\
	\newline Freiburg\\ Germany}
\email{xuwen.zhang@math.uni-freiburg.de}

\begin{abstract}
In this paper we prove the following Liouville-type theorem: any anisotropic minimal graph with free boundary in the half-space must be flat, provided that  the graph function has at most one-sided linear growth. This extends the classical results of Bombieri-De Giorgi-Miranda \cite{BDeGM69} and Simon \cite{Simon76} to an appropriate free boundary setting.

\

\end{abstract}

\maketitle

\section{Introduction}
Let $F:\mbR^{n+1}\ra\mbR_+$ be a uniformly elliptic integrand.
Namely, $F$ is a one-homogeneous function on $\mbR^{n+1}$ which is positive and   $C^2$ on $\mbS^{n}$, satisfying that $\{F<1\}$ is uniformly convex.
For any orientable $C^2$-hypersurface $\S\subset\mbR^{n+1}$, its \textit{anisotropic area} is given by
\eq{\label{defn:anisotropic-area-hypersurface}
\mcA_F(\S)
\coloneqq\int_\S F(\nu)\rd\mcH^n,
}
where $\nu$ is a unit normal vector field of $\S$. If a hypersurface $\S$ is a critical point of $\mcA_F$, then  $\S$ is called an \textit{anisotropic minimal hypersurface} in $\mbR^{n+1}$. When $F(\xi)=F_{eucl}(\xi)\coloneqq\abs{\xi}$, $\mcA_F$ is the standard area and its critical points are minimal hypersurfaces.
The anisotropic minimal hypersurfaces have received considerable attention, not only because they are natural generalizations of the classical minimal surfaces, but also because they play essential roles in many applied sciences, for example, in the study of crystalline structure of the material, see e.g., \cite{DePDeGG18,FMP10,FM11,DeRT22,DePDeR24,DePDeRL24}.

Of particular interest is the case that $\S$ is a graph of some function $u$ over a domain $\Om\subset\mbR^n$.
More precisely, let $u$ be a $C^2$-function defined on $\Om$. Its graph given by
$\S\coloneqq\left\{(x,u(x)):x\in\Om\right\}$ is a $C^2$ embedded hypersurface in $\Om\times\mbR$.
The upwards-pointing unit normal to $\S$, viewed as a vector field defined on $\Om$, is given by
\eq{
\nu(x)
=\frac{(-Du(x),1)}{\sqrt{1+\abs{Du(x)}^2}},\quad x\in\Om,
}
where $D$ is the Euclidean gradient on $\mbR^n$.
In this case, due to the homogeneity of $F$ 
one can rewrite \eqref{defn:anisotropic-area-hypersurface} as (with orientation always given by the upwards-pointing unit normal)
\eq{\label{defn:mathscr-A}
\mathscr{A}_F(u)
=\int_{\Om}f(Du(x))\rd x
}
where
\eq{\label{defn:f}
f(y)
\coloneqq F(-y,1),\quad\forall y\in\mbR^n.
}
If $u$ is a critical point of $\mathscr{A}_F$, 
then it is direct to check that $u$ solves the Euler-Lagrange equation
\eq{\label{defn:AMSE}
{\rm div}\left(Df(Du(x))\right)=0
\quad\text{on }\Om,
}
where ${\rm div}$ denotes the divergence operator on $\mbR^n$.
Equation \eqref{defn:AMSE}, which will be referred to as the \textit{anisotropic minimal surface equation} (AMSE), is a non-uniformly elliptic equation.
In particular, when $F=F_{eucl}$, it recovers the well-known \textit{minimal surface equation} (MSE)
\eq{\label{defn:MSE}
{\rm div}\left(\frac{Du}{\sqrt{1+\abs{Du}^2}}\right)
=0.
}
When $\Omega=\mathbb{R}^n$,  the Bernstein problem asks whether a solution to MSE \eqref{defn:MSE} over $\mbR^n$ is an affine function. This has been solved affirmatively for $n\le 7$ by Bernstein \cite{Bernstein27}, Fleming\cite{Fleming62}, Almgren\cite{Almgren66}, De Giorgi \cite{DeGiorgi65} and Simons \cite{Simons68} and negatively by Bombieri-de Giorgi-Giusti \cite{BDG69}.  
When $F$ is a general elliptic integrand, the anisotropic Bernstein problem, which asks a similar problem for AMSE \eqref{defn:AMSE}, has been completely solved affirmatively for $n=2, 3$ by Jenkins \cite{Jenkins61} and Simon \cite{Simon77}, and negatively for $n\ge 4$ by
Mooney \cite{Mooney22} and Mooney-Yang \cite{MY24}.
On the other hand, a Liouville-type theorem holds for any dimensions by Bombieri-De Giorgi-Miranda \cite{BDeGM69} and Simon \cite{Simon76}, which says that a solution to AMSE \eqref{defn:MSE} over $\mbR^n$ with an additional assumption of at most one-sided linear growth must be affine.
For stable (anisotropic) Bernstein-type problems, we refer to \cite{SSY75,dCP79,F-C-S80,Pogorelov81,Lin90,White91,CM02,Winklmann05-CV, CL23, CLMS24, Mazet24,Bellettini25, LX25} for old results and recent developments.

For the case when $\Om$ has non-empty boundary, Du-Mooney-Yang-Zhu \cite{DMYZ23} recently studied a Dirichlet boundary value problem to AMSE \eqref{defn:AMSE} and proved that, for convex domain $\Om$, if in addition 
\eq{\label{defn:bdry-DMYZ}
u\mid_{\p\Om}\text{ agrees with a linear function }L,
}
then $u$ must be affine, which extends Edelen-Wang's theorem \cite{EW22} to the anisotropic setting.
Note that such $u$ arises as a critical point of \eqref{defn:mathscr-A} with a specific Dirichlet boundary condition \eqref{defn:bdry-DMYZ}.
It is interesting to see  that this result holds for any dimensions in contrast to the anisotropic Bernstein problem mentioned above.

In this paper, we are interested in AMSE \eqref{defn:AMSE} over
\eq{
\Om=\mbR^n_+=\{x\in\mbR^n:x_1>0\}
}
with a natural boundary condition, which leads to a free boundary type condition.
In view of the variational formula
\eq{
\label{vari_formula1}\frac\rd {\rd t}|_{t=0} \mathscr{A}_F(u+t\xi)=
\int_{\mbR^n_+}\left<Df(Du),D\xi\right>
=-\int_{\mbR^n_+}\div\left(Df(Du)\right)\xi -\int_{\partial \mbR^n_+} \langle Df (Du), e_1\rangle\xi ,
}
where $e_1=(1,0,\ldots,0)$ is the unit inward normal of $\mbR^n_+$ along $\p\mbR^n_+$ and $\xi$ is any $C^1$-function on $\mbR^n_+$,
the natural variational boundary condition is thus
\eq{\label{defn:anisotropic-free-bdry-Intro}
\left<Df(Du(x)),e_1\right>=0,\quad\forall x\in\p\mbR^n_+.
}
From the geometric point of view, it is easy to check that with \eqref{defn:AMSE} and \eqref{defn:anisotropic-free-bdry-Intro}, the graph $\S$ is an {anisotropic minimal hypersurface in $\mbR^{n+1}_+=\{x\in\mbR^{n+1}:x_1>0\}$}, such that along $\p\S\subset\p\mbR^{n+1}_+$, the anisotropic normal of $\S$, given by $\nu_F\coloneqq\bar DF(\nu)$, is perpendicular to the support $\p\mbR^{n+1}_+$, i.e.,
\eq{\label{eq:<nu_F,e_1>=0}
\langle \nu_F, e_1\rangle =0,
}
where   $\bar D$ is the Euclidean gradient on $\mbR^{n+1}$ and $e_1$ is the unit inward normal of $\mbR^{n+1}_+$ along $\p\mbR^{n+1}_+$.
Therefore, we call \eqref{defn:anisotropic-free-bdry-Intro} a \textit{free boundary condition in the anisotropic sense}, and  a graph satisfying  \eqref{defn:AMSE} and \eqref{defn:anisotropic-free-bdry-Intro} an {\it anisotropic minimal graph with anisotropic free boundary} (or simply with free boundary).
When $F=F_{eucl}$, it is clear that \eqref{defn:anisotropic-free-bdry-Intro} is 
\eq{\label{old_free}
\langle \nu, e_1\rangle =0.
}
The geometric meaning of \eqref{old_free} is clear: the graph $\Sigma$ is orthogonal to the support $\p \mbR^{n+1}_+$ in the usual sense. Hence it could also be posed as a good boundary condition in the study of hypersurfaces w.r.t. a general $F$. However, our results show that in the general case \eqref{eq:<nu_F,e_1>=0}
is more natural than \eqref{old_free}.

An important example is the so-called capillary minimal graph in $\mbR^{n+1}_+$, which satisfies \eqref{defn:MSE} (the isotropic \eqref{defn:AMSE}) and the classical capillary boundary condition (w.r.t some $\theta\in(0,\pi)$)
\eq{\label{defn:capillary-bdry-condition-iso}
-\frac{u_1}{\sqrt{1+\abs{Du(x)}^2}}
=\left<\nu(x),e_1\right>
=\cos\theta,\quad\forall x\in\p\mbR^n_+.
}
The study of capillary graphs over bounded domains has a long history, see for example \cite{concusfinn1, concusfinn2, Gerhardt76, simonspruck, Uralcprime}. 
Recently, capillary minimal hypersurfaces over unbounded domains have attracted much attention from mathematicians, see \cite{MP21,HS23,LZZ24,CEL25,PTV25,WWZ25}.

 A capillary minimal graph can be viewed as an anisotropic minimal graph with free boundary. To see this, choose the following uniformly elliptic integrand by subtracting a linear term from $F_{eucl}$:
\eq{
F(z)
\coloneqq\abs{z}-\cos\theta\left<z,e_1\right>,\quad
f(y)
=F(-y,1)
=\sqrt{1+\abs{y}^2}+\cos\theta\left<y,e_1\right>,\quad\forall y\in\mbR^n.
}
Then one can easily 
check that with such a choice of $F$, AMSE \eqref{defn:AMSE} is equivalent to MSE \eqref{defn:MSE}, and the capillary boundary condition \eqref{defn:capillary-bdry-condition-iso} becomes \eqref{defn:anisotropic-free-bdry-Intro}.
Hence we can transform the capillary boundary condition to the free boundary condition in the anisotropic sense.
This trick was first used by De Phillipis–Maggi \cite{DePM15}, and has subsequently been adopted as a standard tool in the study of capillary problems, see e.g., \cite{JWXZ25}. We will use it again later.
We remark that another class of capillary graphs $\S$ over $\Om\subset\mbR^n$ with boundary $\p\S=\p\Om\subset\mbR^n$ has been considered by Colombo-Mari-Rigoli \cite{CMR21}.
Compared to ours, their problem can be regarded as an overdetermined boundary problem.


Our main objective in this paper is the following half-space Liouville-type theorem for anisotropic minimal graphs with free boundary in $\mbR^{n+1}_+$.

\begin{theorem}\label{Thm:Liouville}
Let $F:\mbR^{n+1}\ra\mbR_+$ be a uniformly $C^2$-elliptic integrand,
$u$ be a $C^2$-solution to \eqref{defn:AMSE} on $\mbR^n_+$ satisfying the free boundary condition in the anisotropic sense \eqref{defn:anisotropic-free-bdry-Intro}.
If the negative part of $u$ has  at most linear growth on $\mbR^n_+$, i.e,
\eq{\label{condi:linear-growth-negative-part}
u(x)>-\beta(1+|x|), \quad \forall x\in\overline{\mbR^n_+}
}
for some $\beta\geq0$,
then $u$ must be affine.
\end{theorem}
Theorem \ref{Thm:Liouville} holds for any  $F$ and any dimension $n$, and we point out that the linear growth assumption is sharp, as can be inferred from the recent counterexample constructed by Mooney-Yang \cite{MY24}, which we shall discuss below (see \eqref{example:MY24-superlinear}).

Moreover, thanks to the specialty of the half-space, our Liouville-type theorem in fact holds in a more general setting, that is, for \textit{anisotropic minimal graphs with capillary boundary in $\mbR^{n+1}_+$}.
Such graphs are anisotropic minimal graphs in $\mbR^{n+1}_+$, intersecting $\p\mbR^{n+1}_+$ along their boundaries with a constant contact angle condition in the anisotropic sense as follows:
\eq{\label{defn:capillary-bdry-condition}
\left<\nu_F,e_1\right>
=-\om_0,
}
where $\om_0\in(-F(e_1),F(-e_1))$ is a given constant prescribing the contact angle.
Using again the trick by De Phillipis–Maggi \cite{DePM15}, see also \cite[Section 2.3]{JWXZ25}, this can be transformed to an anisotropic free boundary problem, and hence we obtain:
\begin{corollary}\label{Cor:Liouville-capillary-ani}
Let $F:\mbR^{n+1}\ra\mbR_+$ be a uniformly $C^2$-elliptic integrand and $\om_0\in\left(-F(e_1),F(-e_1)\right)$,
$u$ be a $C^2$-solution to \eqref{defn:AMSE} on $\mbR^n_+$, such that its graph $\S$ satisfies the capillary boundary condition in the anisotropic sense \eqref{defn:capillary-bdry-condition}.
If the negative part of $u$ has at most linear growth on $\mbR^n_+$, then $u$ must be affine.
\end{corollary}
Applying this result to classical capillary minimal graphs in $\mbR^{n+1}_+$ (precisely, letting $F=F_{eucl}$, $\om_0=-\cos\theta$ in Corollary \ref{Cor:Liouville-capillary-ani}), we have thus shown:
\begin{corollary}\label{Cor:Liouville-capillary}
Let $\theta\in(0,\pi)$, $u$ be a $C^2$-solution to \eqref{defn:MSE} on $\mbR^n_+$, such that its graph $\S$ satisfies the capillary boundary condition \eqref{defn:capillary-bdry-condition-iso}.
If the negative part of $u$ has at most linear growth on $\mbR^n_+$, then $u$ must be affine.
\end{corollary}

Corollary \ref{Cor:Liouville-capillary} completely removes a range restriction on $\theta$ about  Liouville-type theorems proved recently by Wang-Wei-Zhang in \cite{WWZ25}.
The key point in \cite{WWZ25} is, by utilizing the maximum principle, to obtain a local gradient estimate under certain angle conditions as follows:
\eq{
\abs{Du(0)}
\le C_1e^{1+C_2\frac{\sup_{E_r}\abs{u}}{r}+C_3\left(\frac{\sup_{E_r}\abs{u}}{r}\right)^2},\quad\forall r>0,
}
where $E_r$ is an ellipsoid-type domain designed to match the capillary boundary condition, $C_1,C_2,C_3$ are constants depending only on $n$ and $\theta$. 
A drawback of using the maximum principle is that it is very difficult to obtain the above estimate for the whole range of $\theta$.
In contrast, the crucial step in proving
Theorem \ref{Thm:Liouville} is, by using the integral method, to obtain the following improved gradient estimate, which implies the desired estimate for the whole range of $\theta$ in the capillary case.
\begin{theorem}\label{Thm:gradient-estimate}
Let $F:\mbR^{n+1}\ra\mbR_+$ be a uniformly $C^2$-elliptic integrand.
Let $u$ be a $C^2$-solution to \eqref{defn:AMSE} on $\mbR^n_+$ satisfying the free boundary condition in the anisotropic sense \eqref{defn:anisotropic-free-bdry-Intro}.
There exist positive constants $\mathscr{C}_1=\mathscr{C}_1(n,F)$, $\mathscr{C}_2=\mathscr{C}_2(n,F)$ depending only on $n$ and $F$, such that
for any $x\in\overline{\mbR^n_+}$, the gradient estimate holds:
\eq{\label{ineq:gradient-estimate-intro}
\abs{Du(x)}
\leq e^{\mathscr{C}_1+\frac{\mathscr{C}_2}r\left(\sup_{B_{r,+}(x)}u-u(x)\right)},\quad\forall r>0,
}
where $B_{r,+}(x)=B_r(x)\cap\mbR^n_+$.
\end{theorem}
For the classical minimal surface equation \eqref{defn:MSE}, Bombieri-De Giorgi-Miranda \cite{BDeGM69} established the interior gradient estimates ($n=2$ already by Finn \cite{Finn54}) by using the nowadays well-known integral methods.
Precisely, they proved that
\eq{\label{esti:interior-gradient-BDM69-intro}
\abs{Du(x)}
\leq e^{C_1+\frac{C_2}r\left(\sup_{B_r(x)}u-u(x)\right)},
}
where $C_1,C_2$ are constants depending only on $n$.
An alternative proof was later shown by Trudinger \cite{Trudinger72}, which is somewhat similar to the classical potential theory.
In \cite{Simon76}, Simon obtained the Bombieri-De Giorgi-Miranda-type interior gradient estimates for a very general class of non-uniformly elliptic equations using integral method, see also \cite{LU70}.
In terms of \eqref{defn:AMSE} with respect to a general uniformly elliptic $F$, Simon's estimates can be simplified (c.f., \cite[Section 4]{Simon76}) as  \eqref{esti:interior-gradient-BDM69-intro}, 
with $C_1,C_2$ additionally depending on $F$.
Our gradient estimate \eqref{ineq:gradient-estimate-intro} can be regarded as the anisotropic free boundary analogue of the classical estimates of Bombieri–De Giorgi–Miranda \cite{BDeGM69} and Simon \cite{Simon76}.


Our results are optimal in the following sense.
 The first example of a non-flat entire minimal graph, a minimal graph defined over the  whole Euclidean space $\mbR^n$,  was constructed in \cite{BDG69}, which behaves like
\eq{
\abs{u(x)}
=O(\abs{x}^{2+O(\frac1n)}).
}
Very recently
Mooney-Yang  constructed in \cite{MY24}  a pair $(u,F)$ solving \eqref{defn:AMSE} on $\mbR^n$ ($n\geq4$) with
the following superlinear growth\eq{\label{example:MY24-superlinear}
\sup_{B_r}u\sim r^{1+\mu}.
}
for any $\mu\in(0,\frac12)$.
Due to the symmetry in their examples it is easy to see that those examples are also non-flat examples of anisotropic minimal hypersurfaces with free boundary in $\mbR^{n+1}_+=\{x_1>0\}$, satisfying \eqref{example:MY24-superlinear}.
Hence the linear growth assumption
in our Theorem \ref{Thm:Liouville} is sharp.


Recently there are a few developments on stable  capillary minimal hypersurfaces, which we want to mention here.
 Hong-Saturnino \cite{HS23} proved that a stable capillary minimal hypersurface in $\mbR^3_+$ is flat, by using the idea of Fischer-Colbrie and Schoen \cite{F-C-S80}. See also \cite{MP21,LZZ24}.
For higher dimensions,  Li-Zhou-Zhu proved in \cite{LZZ24} that a stable capillary minimal hypersurface with Euclidean volume growth in $\mbR^{n+1}_+$ is flat for $n=3, 4, 5$ provided that the capillary angle $\theta$ belongs to a certain range, by applying methods of Schoen-Simon-Yau \cite{SSY75}. 
For the anisotropic case, there are on Guo-Xia \cite{GX25} proved that a stable anisotropic capillary minimal hypersurface with Euclidean volume growth in $\mbR^3_+$ is flat. However, in  the anisotropic setting,  higher dimensional cases are still unknown. We remark that (anisotropic) capillary minimal graphs are stable (anisotropic) capillary minimal hypersurfaces.



Now we sketch our ideas of proof.
To obtain Theorem \ref{Thm:gradient-estimate} we use the classical integral methods enlightened by Bombieri-De Giorgi-Miranda \cite{BDeGM69}, Trudinger \cite{Trudinger72}, and Simon \cite{Simon76}, as mentioned above.
Note that even for the classical capillary boundary condition \eqref{defn:capillary-bdry-condition-iso}, certain difficulties arise in pursuing the rigidity results.
For example, in \cite{LZZ24}, to apply the classical Schoen-Simon-Yau technique \cite{SSY75}, Li-Zhou-Zhu analyzed the normal derivative of the length of second fundamental form along the boundary, which naturally appears when integrating by parts, see also \cite{CEL25,PTV25} for related results.
Their idea is to use a \textit{trace estimate} in virtue of the capillary boundary condition, to absorb this boundary term into the integral over $\S$.
For two dimension the rigidity result holds for any $\theta\in(0,\pi)$, but for higher dimensions, the analysis becomes delicate, which results in a certain restriction on $\theta$.

In our case, we handle this difficulty by considering the \textit{anisotropic graphical area element} (recalling \eqref{defn:mathscr-A})
\eq{
W_f(x)
\coloneqq f(Du(x)),\quad\forall x\in\overline{\mbR^n_+},
}
which reduces to the standard \textit{graphical area element} $W(x)=\sqrt{1+\abs{Du(x)}^2}$ when $F=F_{eucl}$.
In the capillary case, $W_f$ coincides with the so-called \textit{capillary graphical area element} studied in \cite{WWZ25}. One can write 
\eq{\label{eq:W-W_f-geometric}
W_f=F(\nu)\<\nu, e_{n+1}\>^{-1},
}
which has a better  geometric meaning. 
The starting point of this paper is a new observation  on $W_f$, see Lemma \ref{Lem:bdry-tangential-property}:
\begin{itemize}
    \item 
The geometric condition 
\eqref{eq:<nu_F,e_1>=0} implies the following analytic condition
\eq{\label{eq-bdry-tangential-property-intro}
g\left(\na_F W_f,\mu\right)=0,\quad \hbox{ along } \p\S,
}
where $g$ is the metric of $\S$ induced from embedding in $\mbR^{n+1}$, $\na_F=\bar D^2F(\nu)\circ\na$ is the $F$-anisotropic gradient on $\S$,  $\na$ is the gradient associated to $g$ 
and $\mu$ is the outer unit co-normal to $\p\S$ with respect to $\S$.
\end{itemize}
This property is essential for our purpose since it eliminates boundary terms when integrating by parts.
In the capillary case, this identity was known earlier (see Uraltseva \cite{Uralcprime} and Gerhardt \cite{Gerhardt76}) and can be checked by direct computations in terms of $u$ and its derivatives.
{
In the anisotropic case, we prove this identity in a more geometric way.
Precisely, by virtue of \eqref{eq:W-W_f-geometric}, we may relate the quantities in \eqref{eq-bdry-tangential-property-intro} to the anisotropic second fundamental form (denoted by $h_F$) and the anisotropic co-normal $\mu_F$ (defined by \eqref{eq_umF} below), and the identity follows from the fact that the unit co-normal $\mu$ is a principal direction of $h_F$.
}


As an intermediate step towards the claimed gradient estimate, by using the well-known De Giorgi-Nash-Moser iteration, we prove in Proposition \ref{Prop:Mean-value-ineq} a mean value inequality for $\log W_f$.
This relies on \eqref{eq-bdry-tangential-property-intro} and the following observation:
\begin{itemize}
    \item  
By the fact that $\S$ has vanishing anisotropic mean curvature, we are able to show that $\log W_f$ is 
a subharmonic function, in fact a weighted $F$-subharmonic function on $\S$, in the sense that
\eq{\label{ineq:psi-subharmonic-intro}
{\rm div}_\S(F^2(\nu)\na_F\log W_f)
\geq F^2(\nu)
g\left(\na\log W_f,\na_F\log W_f\right)\geq0,
}
where ${\rm div}_\S$ is the divergence operator on $(\S,g)$.
See Lemma \ref{Lem:diff-ineq-logW_f}.
\end{itemize}
The iteration argument begins with the ($L^2$-) Sobolev inequality on anisotropic minimal graphs with free boundary in $\mbR^{n+1}_+$.
For any non-negative compactly supported Lipschitz function $\varphi$ on $\S$, and any fixed constant $r>0$, Proposition \ref{Prop:Sobolev-ineq} gives
\eq{
\left(\int_\S{\varphi}^\frac{2n}{n-1}\rd\mcH^n\right)^\frac{n-1}n
\leq C_{n,F}\left(\frac1r\int_\S\varphi^2\rd\mcH^n+r\int_\S\abs{\na\varphi}^2\rd\mcH^n\right).
}
The crucial point is that the inequality does not involve boundary terms.
To achieve such an inequality, we use the classical Michael-Simon inequality \cite{Allard72,MS73}
\eq{\label{ineq:MS-intro}
\left(\int_\S{\varphi}^\frac{n}{n-1}\rd\mcH^n\right)^\frac{n-1}n
\leq C_{MS}\left(\int_{\p\S}\varphi\rd\mcH^{n-1}+\int_\S\abs{\na \varphi}\rd\mcH^n+\int_\S \varphi\abs{H}\rd\mcH^n\right),
}
 the stability inequality 
\eq{
    \int_\S\varphi^2\abs{h}^2\rd\mcH^n
    \leq C(F)\int_\S\abs{\na\varphi}^2\rd\mcH^n,
    }
as well as a trace estimate
\eq{\label{ineq:trace-intro}
\int_{\p\S}\varphi\rd\mcH^{n-1}\leq C(F)\int_\S\abs{\na\varphi}\rd\mcH^n.}
{Here $h$ and $H$ are the classical second fundamental form and mean curvature of $\S\subset\mbR^{n+1}$ } respectively.
The trace estimate follows from testing $X=-\varphi(x)e_1$ in the first variation formula for $\S$
(see Proposition \ref{Prop:1st-variation-aniso}): For any $X\in C^1_c(\mathbb{R}^{n+1};\mathbb{R}^{n+1})$,
\eq{
\int_\S F(\nu)\left(\bar{\rm div}X-\frac{\left<\nu,\bar D_{\bar DF(\nu)}X\right>}{F(\nu)}\right)\rd\mcH^n
=\int_\S H_F\left<X,\nu\right>\rd\mcH^n+\int_{\p\S}\left<X,\mu_F\right>\rd\mcH^{n-1},
}
where $\bar{\rm div}$ is the divergence operator on $\mbR^{n+1}$ and $\mu_F$ is given by
\eq{\label{eq_umF}
\mu_F
=\left<\nu_F,\nu\right>\mu-\left<\nu_F,\mu\right>\nu,
}
which we will call {\it anisotropic co-normal} of $\Sigma$.
A crucial observation is that the free boundary condition \eqref{defn:anisotropic-free-bdry-Intro} implies $\langle \mu_F, -e_1 \rangle$ has a positive lower bound, precisely,
\eq{
\left<\mu_F,-e_1\right>
\geq \min_{\mbS^n}F.
}

Using these ingredients, combined with standard tools from geometric measure theory, we obtain the mean value inequality for $\log W_f$. This is the key step in proving the gradient estimate of Theorem \ref{Thm:gradient-estimate}, which in turn leads to the Liouville theorem (Theorem \ref{Thm:Liouville}).
We emphasize that our approach relies crucially on the variational structure of the anisotropic free boundary problem, and therefore does not apply to anisotropic minimal graphs with the standard free boundary (homogeneous Neumann) condition \eqref{old_free}.

\

\noindent{\it The rest of  the paper  is organized as follows.}
In Section \ref{Sec:2} we first collect some useful facts on elliptic integrand and anisotropic geometry.
Then we prove several geometric facts for anisotropic minimal hypersurfaces/graphs with free boundary (in $\mbR^{n+1}_+$).
In Section \ref{Sec:3} we establish $L^1$-integral estimates for $\log W_f$.
In Section \ref{Sec:4} we prove the mean value inequality.
In Section \ref{Sec:5} we prove Theorem \ref{Thm:gradient-estimate} and Theorem \ref{Thm:Liouville}.

\

\noindent{\it Acknowledgments.}
This work was carried out while
W. Wei was visiting University of Freiburg supported  by the Alexander von Humboldt research fellowship. She would like to thank Institute of Mathematics, University of Freiburg for its  hospitality. 
She is also
partially supported by NSFC (Grant No. 12201288, 12571218, 11771204).
C. Xia is supported by NSFC (Grant No. 12271449, 12126102) and the Natural Science Foundation of Fujian Province of China (Grant No. 2024J011008).

\section{Preliminaries}\label{Sec:2}

Let $\bar D, \bar{\rm div},\left<\cdot,\cdot\right>$ denote the Euclidean gradient, divergence, and scalar product on $\mbR^{n+1}$; and $D,\,\, {\rm div}$ denote the Euclidean gradient, divergence on $\mbR^n$.
We use $B_r(x)$ to denote the Euclidean ball of radius $r$ and center $x\in\mbR^n$ in $\mbR^n$.
For $p\in\mbR^{n+1}$, we use $B^{n+1}_r(p)$ to denote the Euclidean ball of radius $r$ and center $p$ in $\mbR^{n+1}$.
We denote by ${x_1,\ldots,x_{n+1}}$ the Cartesian coordinates of $\mathbb{R}^{n+1}$, and by $e_i=(0,\ldots,0,1,0,\ldots,0)$ the standard basis vector with $1$ in the $i$-th position, for $i=1,\ldots,n+1$.

For a $C^2$-function $s$ on $\mbR^n$, we adopt the conventions $s_i=\frac{\p s}{\p x_i}$, $s_{ij}=\frac{\p^2s}{\p x_i\p x_j}$ for each $1\leq i,j\leq n$.
Note that \eqref{defn:AMSE} can be written as
\eq{\label{defn:AMSE-f_ij-u_ij}
f_{ij}(Du(x))u_{ij}(x)=0\text{ in }\mbR^n_+
}
with the Einstein convention.

\subsection{Elliptic integrand}
Let $F:\mbR^{n+1}\ra\mbR_+$ be a uniformly elliptic integrand.
Precisely, $F\in C^2(\mbR^{n+1}\setminus\{0\})$ and is positive, one-homogeneous function on $\mbR^{n+1}$ such that $\bar D^2 F^2(z)$ is positive definite for all $z\ne 0$.
We record the following useful properties of $F$:
\begin{itemize}
    \item For any $z\in\mbR^{n+1}\setminus\{0\}$,
    \eq{\label{eq:<DF(z),z>=F(z)}
    \left<\bar DF(z),z\right>=F(z).
    }
    \item If we denote the minimum and the maximum values of $F$ on $\mbS^n$ by
\eq{
m_F
\coloneqq\min_{\mbS^n}F,\quad
M_F
\coloneqq\max_{\mbS^n}F,
}
then we have
\eq{\label{ineq:bar-D-F-min-max}
m_F
\leq\Abs{\bar DF(z)}
\leq M_F,
}
where $\bar D F=(\frac{\partial F}{\partial z_1}, \cdots,\frac{\partial F}{\partial z_{n+1}} )$.
    \item There exist positive constants 
\eq{\label{defn:lambda-Lambda_F}
\Lambda_F
\coloneqq &\sup_{z\in\mbS^n,V\in z^\perp\setminus\{0\}}\frac{\frac{\p^2F}{\p z_\alpha\p z_\beta}(z)V^\alpha V^\beta}{\abs{V}^2},\\
\lambda_F
\coloneqq &\inf_{z\in\mbS^n,V\in z^\perp\setminus\{0\}}\frac{\frac{\p^2F}{\p z_\alpha\p z_\beta}(z)V^\alpha V^\beta}{\abs{V}^2}
>0.
}
\end{itemize}

Note that AMSE \eqref{defn:AMSE} is uniformly elliptic when $Du$ is bounded.
For completeness we include the proof.
\begin{proposition}\label{fF} Let $f(p)=F(-p, 1), p\in \mbR^n$.  Assume that $p$ is uniformly bounded, say $|p|\le C$. Then $D^2 f$ is uniformly elliptic. More precisely, 
$$  \frac{1}{(1+C^2)^2}\lambda_F |\xi|^2\le D^2f|_p(\xi,\xi)\le (1+C^2)\Lambda_F |\xi|^2.$$
\end{proposition}
\begin{proof}
Recall that $\{e_i\}_{i=1,\cdots, n+1}$, is the standard unit basis vector of $\mbR^{n+1}$.
We put 
\begin{eqnarray*}
&&v_i=e_i+p_ie_{n+1}, i=1,\cdots, n, \\ &&v_{n+1}=-\sum\limits_{i=1}^n p_ie_i+e_{n+1}.
\end{eqnarray*}
Then $v_i\perp v_{n+1}$ for all $ i=1,\cdots, n$.
By simple computation, we get
\begin{eqnarray*}
&&e_i=\sum_{k=1}^n\left(\delta_{ik}-\frac{p_ip_k}{1+|p|^2}\right)v_k -p_iv_{n+1}, \quad i=1,\cdots, n,
\\ &&e_{n+1}=\frac{1}{1+|p|^2}\sum\limits_{i=1}^n p_i v_i+v_{n+1}.
\end{eqnarray*}
By $1$-homogeneity of $F$, $$\bar D^2F|_{(-p, 1)}(v_{n+1},\cdot)=0.$$
Hence for $i, j=1,\cdots, n$, we have
\begin{eqnarray*}
D^2f|_p(e_i, e_j)&=&\bar{D}^2F|_{(-p, 1)}(e_i,e_j)
\\&=&\sum_{k, l=1}^n \left(\delta_{ik}-\frac{p_ip_k}{1+|p|^2}\right)\left(\delta_{jl}-\frac{p_jp_l}{1+|p|^2}\right)\bar{D}^2F|_{(-p, 1)}(v_k, v_l)
\\&=&\sum_{k, l=1}^n \left(\delta_{ik}-\frac{p_ip_k}{1+|p|^2}\right)\left(\delta_{jl}-\frac{p_jp_l}{1+|p|^2}\right) \sqrt{1+p_k^2}\sqrt{1+p_l^2}\bar{D}^2F|_{(-p, 1)}\left(\frac{v_k}{|v_k|}, \frac{v_l}{|v_l|}\right).
\end{eqnarray*}
Recall that $$\lambda_F\delta_{kl}\le \bar{D}^2F|_{(-p, 1)}\left(\frac{v_k}{|v_k|}, \frac{v_l}{|v_l|}\right)\le \Lambda_F\delta_{kl}.$$
On the other hand,  if $|p|\le C$, then
$$ \frac{1}{1+C^2}\delta_{ik}\le\delta_{ik}-\frac{p_ip_k}{1+|p|^2}\le \delta_{ik}.$$ 
It follows that
\begin{eqnarray*}
  \frac{1}{(1+C^2)^2}\lambda_F|\xi|^2\le \sum_{i,j=1}^nD^2f|_p(e_i, e_j)\xi_i\xi_j\le (1+C^2)\Lambda_F|\xi|^2.
\end{eqnarray*}
\end{proof}

\subsection{Anisotropic free boundary minimal hypersurfaces}\label{Sec:2-free-bdry-minimal}
We first collect some well-known facts.
Consider any $C^2$-hypersurface embedded in $\mbR^{n+1}_+$ with a chosen unit normal vector field along $\S$, denoted by $\nu$.
Suppose that $\S$ has non-empty boundary $\p\S$, which lies entirely in $\p\mbR^{n+1}_+$, so that $\S$ meets $\p\mbR^{n+1}_+$ transversally along $\p\S$.
We denote by $\mu$  the outer unit co-normal to $\p\S$ with respect to $\S$, and then by $\bar\nu$ the outer unit co-normal to $\p\S$ in $\p\mbR^{n+1}_+$, such that (noting that $-e_1$ is the outer unit normal to $\p\mbR^{n+1}_+$) the bases $\{\nu,\mu\}$ and $\{\bar\nu,-e_1\}$ have the same orientation in the normal bundle of $\p\S$ (as a co-diminsion $2$ submanifold in $\mbR^{n+1}$).
Hence it is easy to see the following relations
\eq{\label{eq:mu-relation}
\mu=&-\left<\nu,-e_1\right>\bar\nu+\left<\mu,-e_1\right>(-e_1),\\
\nu=&\left<\mu,-e_1\right>\bar\nu+\left<\nu,-e_1\right>(-e_1).
}
Equivalently,
\eq{\label{eq:bar-nu-relation}
\bar\nu=&-\left<\nu,-e_1\right>\mu+\left<\mu,-e_1\right>\nu,\\
-e_1=&\left<\mu,-e_1\right>\mu+\left<\nu,-e_1\right>\nu.
}

Let $g$ denote the metric on $\Sigma$ induced from embedding in $\mbR^{n+1}$.
Denote by $\na$ the  Levi-Civita connection and  by ${\rm div}_\S$ the  divergence operator on  $(\S,g)$.
We use $h$ to denote the second fundamental form of $(\S,g)\subset\mbR^{n+1}$, i.e.,
\eq{
h(X,Y)
=\left<\bar D_X\nu,Y\right>,\quad\forall X,Y\in T\S.
}

Let $\nu_F=\bar DF(\nu)$ be the \textit{anisotropic normal} of $\S$ and 
$A_F=\bar D^2F(\nu)$, where $\nu$ is the upwards pointing unit normal along $\S$.
Since $F$ is 1-homogeneous, it is well-known that
\eq{
\bar D^2F(\nu)(\cdot,\nu)=0,
}
and hence $A_F$ can be identified as a positive definite symmetric $2$-tensor on $\S$.
Hence, we define $h_F\coloneqq A_F\circ h$ as the \textit{anisotropic second fundamental form} on $\S$; and $H_F\coloneqq{\rm tr}_g(h_F)$ as the \textit{anisotropic mean curvature}.
Note also that $h_F$ satisfies
\eq{\label{defn:h_F}
h_F(X,Y)
=\left<\bar D_X\nu_F,Y\right>
=(A_F\circ h)(X,Y),\quad\forall X,Y\in T\S.
}

Let $A^{-1}_F$ denote the inverse of $A_F$, the \textit{$F$-anisotropic metric} on $\S$ is then defined as
\eq{
g_F(X,Y)
=g(A^{-1}_FX,Y),\quad X,Y\in T\S.
}
The \textit{$F$-anisotropic gradient} of a function $\phi$ defined on $\S$ is given by
\eq{
\na_F\phi
=A_F\nabla \phi
=\bar D^2F(\nu)\na\phi,
}
so that the length of $\na_F\phi$ under the $F$-anisotropic metric is given by 
\eq{
g_F\left(\na_F\phi,\na_F\phi\right)
=g(\na\phi,\na_F\phi)=\langle \na \phi, \bar{D}^2 F(\nu) \na \phi\rangle.
}
The \textit{$F$-Laplacian} is then given by
\eq{
\De_F\phi
={\rm div}_\S(\na_F\phi)
={\rm div}_\S\left(\bar D^2F(\nu)\na\phi\right).
}

\begin{definition}\label{Defn:anisotropic-free-bdry}
\normalfont
We call $\S$ an anisotropic minimal hypersurface if $H_F=0$ on $\S$.
We say that $\S$ has free boundary in the anisotropic sense in $\mbR^{n+1}_+$, if $\left<\nu_F,e_1\right>=0$ on $\p\S\subset\p\mbR^{n+1}_+$.
\end{definition}

\begin{lemma}[{\cite[Proposition 4.1]{GX25}}]\label{Lem:principal-direction-mu}
Let $\S$ be a hypersurface having free boundary in the anisotropic sense in $\mbR^{n+1}_+$, then the outer unit co-normal $\mu$ is an anisotropic principal direction of $\p\S$ in $\S$, i.e., $h_F(\mu,\tau)=0$ for any $\tau\in T(\p\S)$. 
\end{lemma}
\begin{proof}
Using \eqref{eq:bar-nu-relation} we compute
\eq{
0
=\na_\tau\left<\nu_F,e_1\right>
=\left<\bar D_\tau\nu_F,e_1\right>
=&\left<\bar D_\tau\nu_F,\left<\mu,e_1\right>\mu+\left<\nu,e_1\right>\nu\right>\\
=&\left<\mu,e_1\right>\left<\bar D_\tau\nu_F,\mu\right>+\left<\nu,e_1\right>\left<\bar D_\tau\nu_F,\nu\right>
=\left<\mu,e_1\right>h_F\left(\mu,\tau\right).
}
Since $\S$ meets $\p\mbR^{n+1}_+$ transversely, we have $\left<\mu,e_1\right>\neq0$ on $\p\S$, the assertion then follows.
\end{proof}

However, on some occasions it is not easy to use $\mu$, but more convenient to use the following modified one.

\begin{definition}
\normalfont
Let $\S$ be an embedded hypersurface in $\mbR^{n+1}$ with non-empty boundary $\p\S$,
define the \textit{anisotropic co-normal} along $\p\S$ as
\eq{
\mu_F
\coloneqq\left<\nu_F,\nu\right>\mu-\left<\nu_F,\mu\right>\nu,
}
which clearly lies in the normal bundle of $\p\S$.
\end{definition}
It was observed recently in \cite{GX25} that by the definitions of $\mu_F,\nu_F$, and the relation \eqref{eq:bar-nu-relation},  free boundary in the anisotropic sense can be characterized in terms of $\mu_F$ as follows:
\begin{lemma}[{\cite[Proposition 3.2]{GX25}}]
Let $\S$ be an embedded hypersurface in $\mbR^{n+1}_+$ with non-empty boundary $\p\S$ lying in $\p\mbR^{n+1}_+$, such that $\S$ meets $\p\mbR^{n+1}_+$ transversally.
Then
along $\p\S$, one has
\eq{\label{eq:mu_F-relation}
\left<\mu_F,\bar\nu\right>
=&-\left<\nu_F,-e_1\right>,\\
\left<\mu_F,-e_1\right>
=&\left<\nu_F,\bar\nu\right>.
}
\end{lemma}
From its definition, it is clear that $\mu_F$ always lies in the normal space of $\p \S$, which is spanned by $\mu$ and $\nu$, while in general the anisotropic normal $\nu_F$ does not. Thus it is more convenient to use $\mu_F$ instead of $\nu_F$ in the anisotropic setting.
\begin{lemma}\label{Lem:mu_F-perpendicular}
A hypersurface has free boundary in the anisotropic sense in $\p\mbR^{n+1}_+$ if and only if 
\eq{
\mu_F\perp\p\mbR^{n+1}_+.
}
\end{lemma}
\begin{proof}
It follows 
from \eqref{eq:mu_F-relation} and the free boundary condition \eqref{eq:<nu_F,e_1>=0} that 
$\left<\mu_F,\bar\nu\right>=0$. On the other hand, $\mu_F$ lies in the normal space of $\p\S$. The assertion follows from the above two facts.
\end{proof}
Similar to the classical co-normal $\mu$, the anisotropic co-normal $\mu_F$ naturally appears in integration by parts over hypersurfaces with boundary, particularly in the first variation of the anisotropic surface energy; see \cite{Rosales23,GX25}.
\begin{proposition}\label{Prop:1st-variation-aniso}
Let $\S$ be a $C^2$-hypersurface embedded in $\mbR^{n+1}$ with possibly non-empty boundary $\p\S$, and choose one of the unit normal vector fields along $\S$, denoted by $\nu$.
For any $X\in C^1_c(\mbR^{n+1};\mbR^{n+1})$, there holds
\eq{\label{eq:1st-variation-anisotropic}
\int_\S F(\nu)
{\rm div}_{\S,F}X\rd\mcH^n
=\int_\S H_F\left<X,\nu\right>\rd\mcH^n+\int_{\p\S}\left<X,\mu_F\right>\rd\mcH^{n-1},
}
where
\eq{
{\rm div}_{\S,F}X
\coloneqq
\bar{\rm div}X-\frac{\left<\nu,\bar D_{\bar DF(\nu)}X\right>}{F(\nu)}
}
is called the $F$-divergence of $X$ with respect to $\S$.
\end{proposition}

\begin{proof}
We include the proof for completeness.

Denote by $\{\tau_1(x),\ldots\tau_n(x)\}$ an orthonormal basis of $T_x\S$, for any $X\in C^1_c(\mbR^{n+1},\mbR^{n+1})$, we observe that the vector field
\eq{
Y\coloneqq\left<\nu_F,\nu\right>X-\left<\nu,X\right>\nu_F
}
is tangential along $\S$, since
\eq{
\left<\left<\nu_F,\nu\right>X-\left<\nu,X\right>\nu_F,\nu\right>
=F(\nu)\left<X,\nu\right>-F(\nu)\left<\nu,X\right>=0.
}
Moreover, we claim that the $F$-divergence of $X$ with respect to $\S$ is exactly given by
\eq{
F(\nu){\rm div}_{\S,F}(X)
={\rm div}_\S Y+\left<\nu,X\right>{\rm div}_\S\nu_F.
}
To see this, we rewrite the RHS as
\eq{\label{eq:integral-formu-RHS}
\left<\nu_F,\nu\right>{\rm div}_\S X+{\rm div}_\S\left(-\left<\nu,X\right>\nu_F\right)+\underbrace{{\rm div}_\S\left(\left<\nu_F,\nu\right>X\right)-\left<\nu_F,\nu\right>{\rm div}_\S X}_{\coloneqq I}+\left<\nu,X\right>{\rm div}_\S\nu_F.
}
Note that by symmetry of the classical second fundamental form $h$, we find
\eq{
I=\sum_{i=1}^n\left<\nu_F,\bar D_{\tau_i}\nu\right>\left<X,\tau_i\right>
=h(X^T,\nu_F^T)
=h(\nu_F^T,X^T)
=\sum_{i=1}^n\left<\bar D_{\tau_i}\nu,X\right>\left<\nu_F,\tau_i\right>.
}
One can also find that
\eq{
\left<\nu_F,-\sum_{i=1}^n\left<\nu,\bar D_{\tau_i}X\right>\tau_i\right>
=&\sum_{i=1}^n\left\{-\left<\bar D_{\tau_i}\left(\left<\nu,X\right>\nu_F\right),\tau_i\right>+\left<\bar D_{\tau_i}\nu,X\right>\left<\nu_F,\tau_i\right>+\left<\nu,X\right>\left<\bar D_{\tau_i}\nu_F,\tau_i\right>\right\}\\
=&{\rm div}_\S\left(-\left<\nu,X\right>\nu_F\right)+\sum_{i=1}^n\left<\bar D_{\tau_i}\nu,X\right>\left<\nu_F,\tau_i\right>+\left<\nu,X\right>{{\rm div}_\S\left(\nu_F\right)}.
}
Hence, plugging the above two formulas into \eqref{eq:integral-formu-RHS}, the RHS can be further rewritten as
\eq{
&\left<\nu_F,\nu\right>{\rm div}_\S X-\sum_{i=1}^n\left<\nu,\bar D_{\tau_i}X\right>\left<\nu_F,\tau_i\right>\\
=&\left<\nu_F,\nu\right>\bar{\rm div}X-\left<\nu,\bar D_{\left<\nu_F,\nu\right>\nu}X\right>-\sum_{i=1}^n\left<\nu,\bar D_{\left<\nu_F,\tau_i\right>\tau_i}X\right>\\
=&\left<\nu_F,\nu\right>\bar{\rm div}X-\left<\nu, \bar D_{\nu_F}X\right>
=F(\nu){\rm div}_{\S,F}(X),
}
proving the claimed fact.

Integrating ${\rm div}_\S Y+\left<\nu,X\right>{\rm div}_\S\nu_F$ over $\S$, then using \eqref{defn:h_F} and the tangential divergence theorem, we deduce
\eq{
\int_\S{\rm div}_\S Y+\left<\nu,X\right>{\rm div}_\S\nu_F\rd\mcH^n
=&\int_\S H_F\left<X,\nu\right>\rd\mcH^n+\int_{\p\S}\left<\left<\nu_F,\nu\right>X-\left<\nu,X\right>\nu_F,\mu\right>\rd\mcH^{n-1}\\
=&\int_\S H_F\left<\nu,X\right>\rd\mcH^n+\int_{\p\S}\left<X,\underbrace{\left<\nu_F,\nu\right>\mu-\left<\nu_F,\mu\right>\nu}_{=\mu_F}\right>\rd\mcH^{n-1},
}
which completes the proof with the claim.
\end{proof}

\subsection{Anisotropic minimal graph}\label{Sec:2-ani-minimal-graph}
Let $u$ be a $C^2$-function on $\mbR^n_+$.
Using a standard calibration argument, one sees that for $u$ being a critical point of $\mathscr{A}_F$ (recalling \eqref{defn:mathscr-A}), its graph $\S$ is a \textit{$F$-stable anisotropic minimal hypersurface with free boundary} in $\mbR^{n+1}_+$, i.e.,
\begin{enumerate}
    \item (minimal) its anisotropic mean curvature satisfies $H_F\equiv0$;
    \item (free boundary) along the boundary $\p\S$, the anisotropic normal $\nu_F$ satisfies
    \eq{
    \left<\nu_F,e_1\right>=0,\quad\text{on }\p\S\subset\p\mbR^{n+1}_+;
    }
    \item (stable) $\S$ is $F$-stable, in the sense that for any $\phi\in C^2_c(\S)$, the stability inequality holds (c.f., \cite[Proposition 3.9]{GX25}):
    \eq{\label{ineq:stability-GX25-1}
    -\int_\S\phi J_F\phi\rd\mcH^n+\int_{\p\S}\left(g\left(\na_F\phi,\mu\right)-q_F\phi\right)\phi
    \geq0,
    }
    where $J_F$ is the \textit{$F$-Jacobi operator} given by
    \eq{
    J_F\phi
    \coloneqq\De_F\phi+{\rm tr}_g\left(A_Fh^2\right)\phi,
    }
    and
    \eq{
    q_F\phi
    =\frac{\left<\nu_F,\mu\right>}{F(\nu)}h_F(\mu,\mu)\phi.
    }
\end{enumerate}
The fact that $\S$ has $H_F\equiv0$ leads to the following well-known facts (c.f., \cite[Theorem 3.1]{Winklmann05-ArchMath}):
\begin{lemma}
Let $u$ be a $C^2$-solution to \eqref{defn:AMSE} on $\mbR^n_+$ (namely, $\S$ has $H_F\equiv0$), then the 
unit normal $\nu$ along $\S$ satisfies
\eq{\label{eq:Jacobi-eqn-minimal-graph}
\De_F\nu
=-{\rm tr}_g\left(A_Fh^2\right)\nu.
}
And the function $F(\nu(x))$ satisfies
\eq{\label{eq:De_F-msF}
\De_F\left(F(\nu)\right)
={\rm tr}_g(h^2_F)-{\rm tr}_g(A_Fh^2)F(\nu).
}
\end{lemma}

Replacing $\phi$ by $F(\nu)\phi$ in the stability inequality \eqref{ineq:stability-GX25-1} yields the following (c.f., \cite[eqn. (5.4)]{GX25}):

\begin{lemma}
Let $u$ be a $C^2$-solution to \eqref{defn:AMSE} on $\mbR^n_+$ satisfying the free boundary condition in the anisotropic sense \eqref{defn:anisotropic-free-bdry-Intro}.
There exists a positive constant $C(F)$ depending only on $F$, such that
for any $\phi\in C^2_c(\S)$,
\eq{\label{ineq:stability-ineq}
    \int_\S\phi^2\abs{h}^2\rd\mcH^n
    \leq C(F)\int_\S\abs{\na\phi}^2\rd\mcH^n.
    }
\end{lemma}
\begin{proof}
By \eqref{eq:De_F-msF} one finds
\eq{
J_F(F(\nu))
={\rm tr}_g(h^2_F).
}
By Lemma \ref{Lem:principal-direction-mu} one computes
    \eq{\label{boundary1}
    \left<\na_F\left(F(\nu)\right),\mu\right>
    =h_F\left(\nu_F,\mu\right)
    =q_F\left(F(\nu)\right)\quad\text{on }\p\S.
    }
Using these facts, and
replace $\phi$ by $F(\nu)\phi$ in the stability inequality \eqref{ineq:stability-GX25-1}, we find: the term involving $J_F$ is
\eq{
&-\int_\S F(\nu)\phi J_F(F(\nu)\phi)\rd\mcH^n\\
=&-\int_\S F(\nu)\phi\left[\phi J_F(F(\nu))+2g\left(\na\phi,\na_F\left(F(\nu)\right)\right)+F(\nu)\De_F\phi\right]\rd\mcH^n\\
=&-\int_\S F(\nu)\phi^2{\rm tr}_g(h^2_F)+\frac12g\left(\na(F^2(\nu)),\na_F\phi^2\right)+F^2(\nu)\phi\De_F\phi\rd\mcH^n\\
=&-\int_\S F(\nu)\phi^2{\rm tr}_g(h^2_F)+F^2(\nu)\phi\De_F\phi-\frac12F^2(\nu)\De_F(\phi^2)\rd\mcH^n-\frac12\int_{\p\S}F^2(\nu)g\left(\na_F\phi^2,\mu\right)\rd\mcH^{n-1}\\
=&-\int_\S F(\nu)\phi^2{\rm tr}_g(h^2_F)-F^2(\nu)g\left(\na\phi,\na_F\phi\right)\rd\mcH^n-\int_{\p\S}F^2(\nu)\phi g\left(\na_F\phi,\mu\right)\rd\mcH^{n-1};
}
while  by \eqref{boundary1} the term involving $q_F$ is
\eq{
&\int_{\p\S}F(\nu)\phi\bigg(g\left(\na_F(F(\nu)\phi),\mu\right)-q_FF(\nu)\phi\bigg)\rd\mcH^{n-1}\\
=&\int_{\p\S}F(\nu)\phi\bigg(\phi g\big(\na_F(F(\nu)),\mu\big)-q_FF(\nu)\phi+F(\nu)g\left(\na_F\phi,\mu\right)\bigg)\rd\mcH^{n-1}\\
=&\int_{\p\S}F^2(\nu)\phi g\left(\na_F\phi,\mu\right)\rd\mcH^{n-1}.
}
Combining, the boundary terms cancel out and we get
\eq{
m_F\int_\S\phi^2{\rm tr}_g(h^2_F)\rd\mcH^n
\leq\int_\S F(\nu)\phi^2{\rm tr}_g(h^2_F)\rd\mcH^n
\leq\int_\S F^2(\nu)g\left(\na\phi,\na_F\phi\right)\rd\mcH^n
\leq M^2_F\int_\S g\left(\na\phi,\na_F\phi\right)\rd\mcH^n.
}
Hence, by \eqref{defn:lambda-Lambda_F},
\eq{
\int_\S\phi^2\abs{h}^2\rd\mcH^n
    {\leq} C(F)\int_\S\phi^2{\rm tr}_g\left(h_F^2\right)\rd\mcH^n
    \leq C(F)\int_\S\abs{\na\phi}^2\rd\mcH^n.
}
\end{proof}

We consider two kinds of area elements, graphical area element $W$ and anisotropic graphical area element $W_f$, which are comparable:
\begin{lemma}
At each $x\in\overline{\mbR^n_+}$, one has
\eq{\label{ineq:comparable}
M_F^{-1}W_f\leq W\leq m_F^{-1}W_f.
}
\end{lemma}
\begin{proof}
By the definition of $f$, the one-homogeneity of $F$, we find
\eq{\label{eq:W_f(x)}
W_f(x)
=f(Du(x))
=F(-Du(x),1)
=F(\nu(x))W(x),
}
the assertion then follows.
\end{proof}
We need the following two important ingredients, the first is due to
$H_F\equiv0$:
\begin{lemma}\label{Lem:diff-ineq-logW_f}
Let $u$ be a $C^2$-solution to \eqref{defn:AMSE} on $\mbR^n_+$, then on $\S$, there holds:
\eq{
\De_F\log W_f+2g\left(\na\log F(\nu),\na_F\log W_f\right)
=&\frac{{\rm tr}_g(h^2_F)}{F(\nu)}
+g\left(\na\log W_f,\na_F\log W_f\right) \\
\geq & g\left(\na\log W_f,\na_F\log W_f\right)
\geq 0.
}
\end{lemma}
\begin{proof}
By elementary computations
\eq{\label{eq:De-log-v}
\De_F\log h
=-h\De_F(h^{-1})
+g\left(\na\log h,\na_F\log h\right),
\quad\forall h\in C^2(\S).
}

Now for ease of computation we put $v_0(x)\coloneqq\left<\nu(x),e_{n+1}\right>$.
By \eqref{eq:W_f(x)}, we can write $W_f=\frac{F(\nu)}{v_0}$, and hence
\eq{\label{eq:De-v^-1}
\De_F(W_f^{-1})
=\De_F\left(\frac{v_0}{F(\nu)}\right)
=&{\rm div}_\S\left(\frac{\na_F v_0}{F(\nu)}-\frac{v_0\na_F\left(F(\nu)\right)}{F^2(\nu)}\right)\\
=&\frac{\De_F v_0}{F(\nu)}-\frac{2g\left(\na v_0,\na_F\left(F(\nu)\right)\right)}{F^2(\nu)}+\frac{2v_0}{F^3(\nu)}
g\left(\na\left(F(\nu)\right),\na_F\left(F(\nu)\right)\right)
-\frac{v_0\De_F\left(F(\nu)\right)}{F^2(\nu)}.
}
Recalling \eqref{eq:Jacobi-eqn-minimal-graph}, \eqref{eq:De_F-msF}, we readily obtain
\eq{\label{eq:De-vartheta-v_0}
\De_F\left(F(\nu)\right)
=-{\rm tr}_g\left(A_Fh^2\right)F(\nu)+{\rm tr}_g(h^2_F),\quad
\De_Fv_0
=-{\rm tr}_g\left(A_Fh^2\right)v_0.
}
Using \eqref{eq:De-log-v}, \eqref{eq:De-v^-1}, taking also the fact that $W_f=\frac{F(\nu)}{v_0}$ and \eqref{eq:De-vartheta-v_0} into account, we get
\eq{
\De_F\log W_f
=&-\frac{\De_F v_0}{v_0}+\frac{2g\left(\na v_0,\na_F\left(F(\nu)\right)\right)}{v_0F(\nu)}
-\frac2{F^2(\nu)}
g\left(\na\left(F(\nu)\right),\na_F\left(F(\nu)\right)\right)\\
&+\frac{\De_F\left(F(\nu)\right)}{F(\nu)}+
g\left(\na\log W_f,\na_F\log W_f\right)\\
=&{\rm tr}_g\left(A_Fh^2\right)(1-1)+\frac{{\rm tr}_g(h^2_F)}{F(\nu)}+\frac{2g\left(\na v_0,\na_F\left(F(\nu)\right)\right)}{v_0F(\nu)}\\
&-\frac2{F^2(\nu)}
g\left(\na\left(F(\nu)\right),\na_F\left(F(\nu)\right)\right)+
g\left(\na\log W_f,\na_F\log W_f\right)\\
=&\frac{{\rm tr}_g(h^2_F)}{F(\nu)}
-2g\left(\na\log F(\nu),\na_F\log W_f\right)+
g\left(\na\log W_f,\na_F\log W_f\right),
}
where we have used in the third equality the following trivial fact resulting from $W_f=\frac{F(\nu)}{v_0}$:
\eq{
\frac{\na v_0}{v_0}
=\frac{W_f}{F(\nu)}\na \left(\frac{F(\nu)}{W_f}\right)
=\frac{W_f}{F(\nu)}\left(\frac{\na\left(F(\nu)\right)}{W_f}-\frac{F(\nu)\na W_f}{W_f^2}\right)
=\na\log F(\nu)-\na\log W_f.
}
The assertion then follows after rearranging.
\end{proof}
Hence on $\S$, if we consider the quantity (which is a weighted $F$-Laplacian)
\eq{
{\rm div}_\S\left(F^2(\nu)\na_F\phi\right)
=F^2(\nu)\left(\De_F\phi+2g\left(\na\log F(\nu),\na_F\phi\right)\right),
\quad\forall\phi\in C^2(\S),
}
then in view of Lemma \ref{Lem:diff-ineq-logW_f}, we have the point-wise differential inequality on $\S$:
\eq{\label{eq:laplace-log-v}
{\rm div}_\S(F^2(\nu)\na_F\log W_f)
\geq F^2(\nu)
g\left(\na\log W_f,\na_F\log W_f\right).
}
Clearly, for the (weight) function $F^2(\nu)$ defined on $\overline{\mbR^n_+}$,
we have:
\eq{\label{eq:psi-range}
F^2(\nu)\in\left[m_F^2,M_F^2\right].
}

The second one is due to the free boundary condition in the anisotropic sense:
\begin{lemma}\label{Lem:bdry-tangential-property}
Let $u$ be a $C^2$-solution on $\mbR^n_+$ such that its graph $\S$ has free boundary in the anisotropic sense in $\mbR^{n+1}_+$.
Then the anisotropic graphical area element $W_f$ satisfies
\eq{\label{eq:na-v-mu=0}
g\left(\na_FW_f,\mu\right)
\equiv0\text{ on }\p\S.
}
\end{lemma}
\begin{proof}
We first use \eqref{eq:<DF(z),z>=F(z)} to rewrite \eqref{eq:W_f(x)} as
\eq{\label{eq:W_f}
W_f(x)
=\left<\nu_F(x),\nu(x)\right>\left<\nu,e_{n+1}\right>^{-1}.
}
In view of Lemma \ref{Lem:principal-direction-mu}
and  Lemma \ref{Lem:mu_F-perpendicular},
we have
\eq{
g(\na_FW_f,\mu)
=&\left<A_F[\na(W_f)],\mu\right>
=\left<A_F\left[\na\left(\left<\nu_F(x),\nu(x)\right>\left<\nu,e_{n+1}\right>^{-1}\right)\right],\mu\right>\\
=&0+h_F\left(\nu_F,\mu\right)\left<\nu,e_{n+1}\right>^{-1}-\left<\nu_F,\nu\right>\left<\nu,e_{n+1}\right>^{-2}h_F(\mu,e_{n+1})\\
=&\left<\nu,e_{n+1}\right>^{-2}h_F\left(\mu,\mu\right)\left(\left<\nu,e_{n+1}\right>\left<\nu_F,\mu\right>-\left<\nu_F,\nu\right>\left<\mu,e_{n+1}\right>\right)\\
=&\left<\nu,e_{n+1}\right>^{-2}h_F(\mu,\mu)\left<
{\left<\nu_F,\mu\right>\nu-\left<\nu_F,\nu\right>\mu},e_{n+1}\right>
\\
=&-\left<\nu,e_{n+1}\right>^{-2}h_F(\mu,\mu) \langle{\mu_F},e_{n+1}\rangle=0.
}
\end{proof}

Finally, note that
any $C^1$-function $\phi$ on $\mbR^n$ can be viewed as a function on $\S$ by putting $\phi(x,u(x))=\phi(x)$.
Hence
\eq{
\abs{\na\phi}^2
=\abs{D\phi}^2-W^{-2}\abs{\left<Du,D\phi\right>}^2
\geq\abs{D\phi}^2-W^{-2}\abs{Du}^2\abs{D\phi}^2
=W^{-2}\abs{D\phi}^2.
}
Recalling \eqref{defn:lambda-Lambda_F}, we find
\eq{\label{ineq:bar-nabla-D-relation}
\Lambda_F\abs{D\phi}^2
\geq \Lambda_F\abs{\na\phi}^2
\geq
g(\na\phi,\na_F\phi)
\geq \lambda_F\abs{\na\phi}^2
\geq \lambda_FW^{-2}\abs{D\phi}^2.
}
Similarly, any $C^1$-function $\psi$ on $\mbR^{n+1}$ when restricted to $\S$ satisfies
\eq{\label{ineq:bar-nabla-D-relation-n+1}
\Lambda_F\abs{\bar D\psi}^2
\geq \Lambda_F\abs{\na\psi}^2
\geq
g(\na\psi,\na_F\psi)
\geq \lambda_F\abs{\na\psi}^2.
}

\section{Integral estimates}\label{Sec:3}
The goal of this section is to obtain the integral estimates (Lemmas \ref{Lem:integral-estimate-2}, \ref{Lem:integral-estimate-3}), inspired by Trudinger \cite{Trudinger72}.

\begin{remark}\label{Rem:tilde-F}
\normalfont
    For general $F$, one may have $W_f(x)<1$ on $\overline{\mbR^n_+}$, so $\log W_f(x)<0$. To avoid this, set
    \eq{
    \tilde F(z)
    \coloneqq m^{-1}_FF(z),\quad\forall z\in\mbR^{n+1}\setminus\{0\},
    }
    so that $m_{\tilde F}=1$. With $\tilde f(y)=\tilde F(-y,1)$ we have $D\tilde f=m_F^{-1}Df$, hence any $C^2$ solution $u$ of \eqref{defn:AMSE} with free boundary condition \eqref{defn:anisotropic-free-bdry-Intro} for $f$ also satisfies it for $\tilde f$. Thus all results of Section \ref{Sec:2-ani-minimal-graph} remain valid with $\tilde F$. Moreover,
    \eq{
    W_{\tilde f}(x)
    =\tilde f(Du(x))
    =\tilde F(\nu(x))\sqrt{1+\abs{Du(x)}^2}
    \geq1,\quad\forall x\in\overline{\mbR^n_+},
    }
    so $\log W_{\tilde f}\ge 0$ on $\overline{\mbR^n_+}$.
\end{remark}
Having in mind Remark \ref{Rem:tilde-F}, in all follows (until the end of Section \ref{Sec:5}), we may assume without loss of generality that the uniformly elliptic integrand $F$ satisfies $m_F\geq1$ (so that $\log W_f\geq0$), otherwise we could simply replace $F$ by $\tilde F$.

Now for any $x_0\in\overline{\mbR^n_+}$ fixed, we define a Lipschitz cut-off function $\eta$ on $\mbR^n_+$ with compact support, satisfying $\eta\equiv1$ on $B_{r,+}(x_0)$, $\eta\equiv0$ outside $B_{2r,+}(x_0)$, and $\abs{D\eta}\leq\frac2r$.

\begin{lemma}\label{Lem:integral-estimate-1}
Let $u$ be a $C^2$-solution to \eqref{defn:AMSE} on $\mbR^n_+$ satisfying the free boundary boundary condition in the anisotropic sense \eqref{defn:anisotropic-free-bdry-Intro}.
Then for any $x_0\in\overline{\mbR^n_+}$,
\eq{\label{eq:Ding24-(3.1)}
\int_{B_{2r,+}(x_0)\cap\left\{u>-r\right\}}\abs{D(\log W_f)}\eta
\leq C(n,F)r^{-1}\int_{B_{2r,+}(x_0)\cap\left\{u>-2r\right\}}W_f,
}
where $\eta$ is given above.
\end{lemma}
\begin{proof}
Let $\xi$ be a compactly supported Lipschitz function on $\mbR^{n+1}$.
By \eqref{eq:laplace-log-v}, then using integration by parts in conjunction with Lemma \ref{Lem:bdry-tangential-property}, we find
\eq{
&\int_\S\xi^2
g\left(\na\log W_f,\na_F\log W_f\right)F^2(\nu)
\leq\int_\S\xi^2{\rm div}_\S\left(F^2(\nu)\na_F\log W_f\right)
=-2\int_\S\xi g(\na\xi,\na_F\log W_f)F^2(\nu)\\
\leq&\frac12\int_\S\xi^2
g\left(\na\log W_f,\na_F\log W_f\right)F^2(\nu)
+2\int_\S
g\left(\na\xi,\na_F\xi\right)F^2(\nu),
}
which implies
\eq{
\int_\S
\xi^2g\left(\na\log W_f,\na_F\log W_f\right)F^2(\nu)
\leq4\int_\S
g\left(\na\xi,\na_F\xi\right)F^2(\nu).
}
Choose $\xi(x,x_{n+1})\coloneqq\eta(x)\tau(x_{n+1})$, where $\eta$ is given above and $\tau$ is a cut-off function on $\mbR^1$ with $0\leq\tau\leq1$, $\tau\equiv1$ in $(-r,\sup_{B_{2r,+}(x_0)}u)$, $\tau\equiv0$ outside $(-2r,r+\sup_{B_{2r,+}(x_0)}u)$, and the derivative of $\tau$ satisfies $\abs{\tau'}<\frac4r$. 
By \eqref{ineq:bar-nabla-D-relation-n+1} and area formula, we thus find
\eq{
&m_{F}^2\int_{\S}
\xi^2g\left(\na\log W_f,\na_F\log W_f\right)
\leq\int_{\S}
\xi^2g\left(\na\log W_f,\na_F\log W_f\right)F^2(\nu)\\
\leq& 64\Lambda_{F}M_{F}^2r^{-2}\mcH^n(\S\cap{\rm spt}(\xi))
\leq 64\Lambda_{F}M_{F}^2r^{-2}\int_{B_{2r,+}(x_0)\cap\left\{u\geq-2r\right\}}W.
}
Therefore, by \eqref{ineq:bar-nabla-D-relation} and the H\"older inequality,
\eq{
\int_{B_{2r,+}(x_0)\cap\left\{u>-r\right\}}\abs{D(\log W_f)}\eta
\leq\lambda_{F}^{-\frac12}\int_\S\xi
\sqrt{g\left(\na\log W_f,\na_F\log W_f\right)}
\leq C(n,F)r^{-1}\int_{B_{2r,+}(x_0)\cap\left\{u\geq-2r\right\}}W.
}
Finally, recall that $W\leq m_{F}^{-1}W_f$,
we complete the proof.

\end{proof}

\begin{lemma}\label{Lem:integral-estimate-2}
Let $u$ be a $C^2$-solution to \eqref{defn:AMSE} on $\mbR^n_+$ satisfying the free boundary condition in the anisotropic sense \eqref{defn:anisotropic-free-bdry-Intro}.
Then for any $x_0\in\overline{\mbR^n_+}$,
\eq{\label{ineq:Ding24-(3.7)}
\fint_{B_{r,+}(x_0)\cap\left\{x\in \mathbb{R}^n_+,\abs{u(x)}\leq r\right\}}(\log W_f)W_f
\leq C(n,F)r^{-n}\int_{B_{2r,+}(x_0)\cap\{u\geq-2r\}}W_f.
}

\end{lemma}
\begin{proof}
Let $\eta$ be as above. 
We  consider the truncated function
\eq{
u_r
\coloneqq
\begin{cases}
    2r,\quad&\text{for }u\geq r,\\
    u+r,\quad&\text{for }\abs{u}\leq r,\\
    0,\quad&\text{for }u\leq-r.
\end{cases}
}
Multiplying  \eqref{defn:AMSE} by $u_r (\log W_f) \eta$, integrating by parts and using \eqref{defn:anisotropic-free-bdry-Intro} give 
\eq{\label{eq:Ding24-(3.11)}
0
=&\int_{\mbR^n_+}\left<Df(Du),D\left(u_r(\log W_f)\eta\right)\right>\\
=&\int_{\left\{\abs{u}\leq r\right\}}\left<Df(Du),Du\right>(\log W_f)\eta+\int_{\mbR^n_+}u_r\eta\left<Df(Du),D(\log W_f)\right>
+\int_{\mbR^n_+}u_r(\log W_f)\left<Df(Du),D\eta\right>.
}
Since  $Df(Du)=-DF(-Du,1)$ and \eqref{ineq:bar-D-F-min-max}, we have 
\eq{\label{esti:Df(Du)-length}
\abs{Df(Du)}\leq\abs{\bar DF(-Du,1)}
\leq M_{F}.
}
Thus for all $r>0$,
 the last two terms in \eqref{eq:Ding24-(3.11)} can be estimated by
\eq{\label{eq:Ding24-(3.11')}
&\int_{\mbR^n_+}u_r\eta\left<Df(Du),D(\log W_f)\right>
+\int_{\mbR^n_+}u_r(\log W_f)\left<Df(Du),D\eta\right>\\
\geq&-2M_{F}r\int_{\left\{u>-r\right\}}\Abs{D(\log W_f)}\eta-{4M_{F}}\int_{B_{2r,+}(x_0)\cap\left\{u>-r\right\}}(\log W_f).
}
Therefore, by virtue of \eqref{eq:Ding24-(3.11)} and \eqref{eq:Ding24-(3.11')}, we obtain
\eq{\label{equality}
\int_{\left\{\abs{u}\leq r\right\}}\left<Df(Du),Du\right>(\log W_f)\eta
\leq 2M_{F} r\int_{\left\{u>-r\right\}}\Abs{D(\log W_f)}\eta+{4M_{F}}\int_{B_{2r,+}(x_0)\cap\left\{u>-r\right\}}(\log W_f).
}

Since $F$ is 1-homogeneous, we find 
\eq{
F(\nu)
=\left<\bar DF(-Du,1),\nu\right>
=W^{-1}\left<DF(-Du,1),-Du\right>+W^{-1}\left<\bar DF(-Du,1),e_{n+1}\right>.
}
Combining with \eqref{ineq:bar-D-F-min-max} and \eqref{ineq:comparable}, we thus obtain
\eq{\label{esti:<Df(Du),Du>}
M^{-1}_{F}m_{F}W_f
\leq m_{F}W
\leq\left<Df(Du),Du\right>+M_{F}.
}
From this and \eqref{equality} we deduce that
\eq{\label{eq:Ding24-(3.12)}
&\int_{\left\{\abs{u}\leq r\right\}}(\log W_f)W_f\eta
\overset{\eqref{ineq:comparable}}{\leq} M_{F}\int_{\left\{\abs{u}\leq r\right\}}(\log W_f)W\eta
\leq M_{F}\int_{\left\{\abs{u}\leq r\right\}}m^{-1}_{F}\left(\left<Df(Du),Du\right>+M_{F}\right)(\log W_f)\eta\\
\leq&\frac{M^2_{F}}{m_{F}}\left(2r\int_{\left\{u>-r\right\}}\Abs{D(\log W_f)}\eta+5\int_{B_{2r,+}(x_0)\cap\left\{u\geq-r\right\}}(\log W_f)\right).
}
By \eqref{eq:Ding24-(3.1)}, we have (noting that $\log W_f\leq W_f$)
\eq{
\int_{B_{r,+}(x_0)\cap\left\{\abs{u}\leq r\right\}}(\log W_f)W_f
\leq C(n,F)\int_{B_{2r,+}(x_0)\cap\{u\geq-2r\}}W_f,
}
which completes the proof.

\end{proof}

\begin{lemma}\label{Lem:integral-estimate-3}
Let $u$ be a $C^2$-solution to \eqref{defn:AMSE} on $\mbR^n_+$ satisfying the free boundary boundary condition in the anisotropic sense \eqref{defn:anisotropic-free-bdry-Intro}.
There exist positive constants $\mfc_1=\mfc_1(n,F)$ and $\mfc_2=\mfc_2(n,F)$ such that for any $x_0\in\overline{\mbR^n_+}$,
\eq{
r^{-n}\int_{B_{r,+}(x_0)\cap\left\{u>-r\right\}}W_f
\leq\left(\mfc_1+\frac{\mfc_2}r\sup_{B_{2r,+}(x_0)}u\right),\quad\forall r>0.
}
\end{lemma}
\begin{proof}
For any Lipschitz function $\xi$ with compact support on $\mbR^n$, by \eqref{defn:AMSE} and \eqref{defn:anisotropic-free-bdry-Intro} we have
\eq{
0
=\int_{\mbR^n_+}\left<Df(Du),D\xi\right>.
}
Letting $\xi\coloneqq\eta\max\{u+r,0\}$, where $\eta$ is given above, 
we have
\eq{\label{eq_a1} 
-\int\left<Df(Du),D\eta\right>\max\{u+r,0\}
=\int_{\left\{u\geq-r\right\}}\eta\left<Df(Du),Du\right>.
}
By \eqref{esti:<Df(Du),Du>} and \eqref{eq_a1}, we obtain
\eq{
&m_{F}\int_{B_{r,+}(x_0)\cap\left\{u>-r\right\}}W_f
\leq m_{F}\int_{\{u>-r\}}W_f\eta
\leq M_{F}\int_{\{u>-r\}}\left<Df(Du),Du\right>\eta+M_{F}^2\int_{\{u>-r\}}\eta\\
\leq&M_{F}^2\int_{B_{2r,+}(x_0)\cap\left\{u\geq-r\right\}}1 +\frac2r M_{F}\int_{B_{2r,+}(x_0)\cap\left\{u\geq-r\right\}}(u+r)\abs{Df(Du)}.
}
In view of  \eqref{esti:Df(Du)-length}, we deduce as required
\eq{
\int_{B_{r,+}(x_0)\cap\left\{u>-r\right\}}W_f
\leq M^2_{F}m_{F}^{-1}r^n\left(c_1(n)+\frac{c_2(n)}r\sup_{B_{2r,+}(x_0)}u\right).
}
\end{proof}

\section{Mean value inequality}\label{Sec:4}
\subsection{Area growth estimates}
Let $u$ be a $C^2$-solution to \eqref{defn:AMSE} on $\mbR^n_+$ satisfying the free boundary condition in the anisotropic sense \eqref{defn:anisotropic-free-bdry-Intro}.
Fix any $x_0\in\overline{\mbR^n_+}$,
assume $u(x_0)=0$,
so that $(x_0,0)\in\S$.
With a slight abuse of notation, we do not distinguish $x_0$, $(x_0,0)$, but simply write $x_0\in\S\subset\mbR^{n+1}$.

For any $r>0$, denote by $B^{n+1}_{2r,+}(x_0)$ the truncated ball of radius $2r$ in $\mbR^{n+1}_+$, centered at $x_0$.
We put
\eq{
\Om_+\coloneqq\left\{(x,t)\in B_{2r,+}(x_0)\times\mbR:t>u(x)\right\},\\
\Om_-\coloneqq\left\{(x,t)\in B_{2r,+}(x_0)\times\mbR:t<u(x)\right\}.
}

\begin{lemma}\label{Lem:volume-lower-bound}
There exists a positive constant $\mfc_\star=\mfc_\star(n,F)$ such that for any $r>0$, we have
\eq{
\mcH^{n+1}\left(\Om_+\cap B^{n+1}_r(x_0)\right)
\geq \mfc_\star\mcH^{n+1}(B^{n+1}_{\frac{r}2,+}(x_0)),
} where $\mcH^{n+1}$ is the $(n+1)$-dimensional Hausdorff measure.
\end{lemma}

\begin{proof}
For any $0<s<r$, we put $\Om_{s}\coloneqq\Om_+\cap\p_{rel} B^{n+1}_{s,+}(x_0)$, where for a set $E\subset\mbR^{n+1}_+$, $\p_{rel}E$ denotes its relative boundary in $\mbR^{n+1}_+$, i.e., $\p_{rel}E=\overline{\p E\cap\mbR^{n+1}_+}$.
Clearly, we have
\eq{
\p_{rel}\Om_{s}
=\p_{rel}\Om_+\cap\p B^{n+1}_s(x_0)
=\p_{rel}\left(\p_{rel}\Om_+\cap B^{n+1}_s(x_0)\right),
}
i.e., $\Omega_s$ and $\p_{rel}\Om_+\cap B^{n+1}_s(x_0)$ have a common relative boundary in $\mbR^{n+1}_+$.
Notice also that $\p_{rel}\Om_+\cap B^{n+1}_s(x_0)$ is a minimizer of the anisotropic area \eqref{defn:anisotropic-area-hypersurface},
we thus find
\eq{
M_F\mcH^n(\Om_s)
\geq\int_{\Om_s}F(\nu)\rd\mcH^n
\geq\int_{\p_{rel}\Om_+\cap B^{n+1}_s(x_0)}F(\nu)\rd\mcH^n
\geq m_F\mcH^n(\p_{rel}\Om_+\cap B^{n+1}_s(x_0)),
}
that is,
\eq{\label{upper bound of ball}
\mcH^n(\p_{rel}\Om_+\cap B^{n+1}_s(x_0))
\leq M_F m_F^{-1}\mcH^n(\Om_s).
}
By the anisotropic relative isoperimetric inequality in $\mbR^{n+1}_+$ (c.f., \cite[Theorem 1.3]{CRS16}), we have
\eq{
\mcH^{n+1}(\Om_+\cap B^{n+1}_{s}(x_0))^\frac{n}{n+1}
\leq C(n,F)P_{F}(\Om_+\cap B^{n+1}_{s}(x_0);\mbR^{n+1}_+)
\leq C(n,F)\mcH^n\left(\p_{rel}(\Om_+\cap B^{n+1}_{s}(x_0))\right),
}
here $P_{F}(E;\mbR^{n+1}_+)$ denotes the anisotropic perimeter for a measurable set $E$ relative to $\mbR^{n+1}_+$, defined as
\eq{
P_{F}(E;\mbR^{n+1}_+)
\coloneqq\int_{\p^\ast E\cap\mbR^{n+1}_+}F(\nu_E)\rd\mcH^n,
}
where $\p^\ast E$ is the reduced boundary of $E$ and $\nu_E$ is the measure-theoretic outer unit normal to $E$.

To proceed,
since by \eqref{upper bound of ball},
\eq{
\mcH^n\left(\p_{rel}(\Om_+\cap B^{n+1}_{s}(x_0))\right)
=\mcH^n(\Om_s)+\mcH^n(\p_{rel}\Om_+\cap B^{n+1}_s(x_0))
\leq\left(1+M_F m_F^{-1}\right)\mcH^n(\Om_s),
}
we get
\eq{
\mcH^{n+1}(\Om_+\cap B^{n+1}_s(x_0))^\frac{n}{n+1}
\leq C(n,F)\mcH^n(\Om_s).
}
By the co-area formula this implies
\eq{
\left(\int^s_0\mcH^n(\Om_t)\rd t\right)^\frac{n}{n+1}
\leq C(n,F)\mcH^n(\Om_s).
}
Clearly,
 $\mcH^n(\Omega_s)>0$ for each $0<s<r$, since $0\in\S$ (recall that we assume $u(0)=0$).
Therefore, \eq{
\frac{\p}{\p s}\left(\int^s_0\mcH^n(\Om_t)\rd t\right)^\frac1{n+1}
\geq C(n,F),
\quad s\in(0,r).
}
Integrating over $(\frac{r}2,r)$ yields
\eq{
\left(\int^r_0\mcH^n(\Om_t)\rd t\right)^\frac1{n+1}
\geq C(n,F)r
\geq C(n,F)\left(\mcH^{n+1}(B^{n+1}_{\frac{r}2,+}(x_0))\right)^\frac1{n+1}.
}
Using the co-area formula again, we deduce as required that
\eq{
\mcH^{n+1}\left(\Om_+\cap B^{n+1}_r(x_0)\right)
=\int^r_0\mcH^n(\Om_s)\rd s
\geq \underbrace{C(n,F)}_{\coloneqq\mfc_\star}\mcH^{n+1}(B^{n+1}_{\frac{r}2,+}(x_0)).
}
\end{proof}
Similarly, for $\Om_-$ we also have
\eq{
\mcH^{n+1}\left(\Om_-\cap B^{n+1}_r(x_0)\right)
\geq \mfc_\star\mcH^{n+1}(B^{n+1}_{\frac{r}2,+}(x_0)),
\quad\forall r>0.
}
By the relative isoperimetric inequality in truncated balls (c.f., \cite[(2.51)]{DePM15}, either applying to $\Om_+$ or $\Om_-$), and note that $\mcH^{n+1}(B^{n+1}_{\frac{r}2,+}(x_0))\geq\mcH^{n+1}(B^{n+1}_{\frac{r}2,+}(0))=C(n)r^{n+1}$, we thus find
\eq{
\mcH^n(\S\cap B^{n+1}_r(x_0))
\geq C(n)(\mfc_\star\mcH^{n+1}(B^{n+1}_{\frac{r}2,+}(x_0)))^{\frac{n}{n+1}}
\geq\mfC_\star r^n,
}
for some $\mfC_\star=\mfC_\star(n,F)>0$. 
On the other hand, 
by \eqref{upper bound of ball}, we obtain
\eq{
\mcH^n(\S\cap B^{n+1}_r(x_0))
\leq \mfC^\star r^n,
}
for some $\mfC^\star=\mfC^\star(n,F)>0$.

For any $r>0$ fixed, we denote by $\mcB_r(x_0)=\S\cap B^{n+1}_r(x_0)$ the \textit{graph ball} of radius $r$.
Combining the above estimates, we obtain:
\begin{proposition}\label{Prop:graph-ball-EVG}
Let $u$ be a $C^2$-solution to \eqref{defn:AMSE} on $\mbR^n_+$ satisfying the free boundary condition in the anisotropic sense \eqref{defn:anisotropic-free-bdry-Intro}, and suppose that $u(x_0)=0$.
There exist positive constants $\mfC^\star=\mfC^\star(n,F)$, $\mfC_\star=\mfC_\star(n,F)$, such that
\eq{\label{esti:graph-ball}
\mfC_\star r^n
\leq\mcH^n(\mcB_r(x_0))
\leq\mfC^\star r^n,\quad\forall r>0.
}
\end{proposition}

\subsection{Sobolev inequality}

To carry out the iteration argument, the following ($L^2$-) Sobolev inequality on anisotropic minimal graph with free boundary in the half-space is crucial:
\begin{proposition}\label{Prop:Sobolev-ineq}
Let $u$ be a $C^2$-solution to \eqref{defn:AMSE} on $\mbR^n_+$ satisfying the free boundary condition in the anisotropic sense \eqref{defn:anisotropic-free-bdry-Intro}.
There exists a positive constant $C_{n,F}$ (depending only on $n,F$), such that for any non-negative compactly supported Lipschitz function $\varphi$ on $\S$, and any fixed constant $r>0$, there holds
\eq{\label{ineq:Sobolev}
\left(\int_\S{\varphi}^\frac{2n}{n-1}\rd\mcH^n\right)^\frac{n-1}n
\leq C_{n,F}\left(\frac1r\int_\S\varphi^2\rd\mcH^n+r\int_\S\abs{\na\varphi}^2\rd\mcH^n\right).
}
\end{proposition}

To establish it we first prove some useful facts.
\begin{lemma}\label{Lem:mu_F-e_1-lower-bound}
Let $u$ be a $C^2$-solution to \eqref{defn:AMSE} on $\mbR^n_+$ satisfying the free boundary condition in the anisotropic sense \eqref{defn:anisotropic-free-bdry-Intro}.
Then there holds
\eq{
\left<\mu_F,-e_1\right>\geq m_F,\quad\text{ on }\p\S.
}
\end{lemma}
\begin{proof}
We denote by $\mbR^{n-1}$ the $(n-1)$-dimensional Euclidean space spanned by $\{x_2,\ldots,x_n\}$, and denote by $D^\ast$ the Euclidean gradient on $\mbR^{n-1}$.
Clearly $\p\S$ can be viewed as the graph of $u$ over $\mbR^{n-1}$, it follows that $\bar\nu$ (the co-normal to $\p\S$ in $\p\mbR^{n+1}_+$) has the following expression:
\eq{
\bar\nu(x)
=\frac{\left(0,-D^\ast u(x),1\right)}{\sqrt{1+\abs{D^\ast u(x)}^2}}.
}
In view of \eqref{eq:mu_F-relation} we just have to estimate $\left<\nu_F,\bar\nu\right>$.
To this end, we use \eqref{eq:<DF(z),z>=F(z)} to compute
\eq{
F(\nu)
=\left<\bar DF(\nu),\nu\right>
=&\left<\left(F_1(\nu),\ldots,F_{n+1}(\nu)\right),\frac{\left(-u_1,-D^\ast u,1\right)}{\sqrt{1+\abs{Du}^2}}\right>\\
\overset{\eqref{defn:anisotropic-free-bdry-Intro}}{=}&\left<\bar DF(\nu),\frac{\left(0,-D^\ast u,1\right)}{\sqrt{1+\abs{Du}^2}}\right>
=\left<\nu_F,\bar\nu\right>\frac{\sqrt{1+\abs{D^\ast u}^2}}{\sqrt{1+\abs{Du}^2}},
}
and hence
\eq{
\left<\nu_F,\bar\nu\right>
=\frac{\sqrt{1+\abs{Du}^2}}{\sqrt{1+\abs{D^\ast u}^2}}F(\nu)
\geq m_F.
}
We complete the proof.

\end{proof}

With this Lemma, we can establish a useful trace estimate:
\begin{lemma}\label{Lem:trace}
Let $u$ be a $C^2$-solution to \eqref{defn:AMSE} on $\mbR^n_+$ satisfying the free boundary condition in the anisotropic sense \eqref{defn:anisotropic-free-bdry-Intro}.
There exists a positive constant $C$, depending only on $F$, such that
for any non-negative compactly supported Lipschitz function $\varphi$ on $\S$, there holds
\eq{\label{ineq:trace-estimate}
\int_{\p\S}\varphi\rd\mcH^{n-1}
\leq C\int_\S\abs{\na\varphi}\rd\mcH^n.
}
\end{lemma}
\begin{proof}
First we consider non-negative $\varphi\in C^1(\S)$ with compact support, and extend it to a $C^1$-function on $\mbR^{n+1}$, still denoted by $\varphi$, such that $\left<\bar D\varphi,\nu\right>=0$ on $\S$.

Letting $X=-\varphi e_1$ in \eqref{eq:1st-variation-anisotropic}, we obtain
\eq{\label{sobo1}
\int_\S F(\nu)\left(\bar{\rm div}X-\frac{\left<\nu,\bar D_{\bar DF(\nu)}X\right>}{F(\nu)}\right)\rd\mcH^n
=\int_{\p\S}\varphi\left<\mu_F,-e_1\right>\rd\mcH^{n-1}.
}
By direct computations, we  get
\eq{
F(\nu)\left(\bar{\rm div}X-\frac{\left<\nu,\bar D_{\bar DF(\nu)}X\right>}{F(\nu)}\right)
=&-F(\nu)\left<\bar D\varphi,e_1\right>+\left<\nu,e_1\right>\left<\bar DF(\nu),\bar D\varphi\right>\\
=&-F(\nu)\left<\na\varphi,e_1\right>+\left<\nu,e_1\right>\left<\bar DF(\nu),\na\varphi\right>
\leq C(F)\abs{\na\varphi}.
}
Together with Lemma \ref{Lem:mu_F-e_1-lower-bound} and \eqref{sobo1},  we have thus shown
\eq{
\int_{\p\S}\varphi\rd\mcH^{n-1}
\leq C(F)\int_\S\abs{\na\varphi}\rd\mcH^n.
}
For non-negative $\varphi$ which is Lipschitz on $\S$ with compact support, the  estimate holds by approximation.

\end{proof}

\begin{proof}[Proof of Proposition \ref{Prop:Sobolev-ineq}]
By the classical Michael-Simon inequality on $\S$, and the Cauchy-Schwarz inequality:
\eq{
\left(\int_\S{\varphi}^\frac{n}{n-1}\rd\mcH^n\right)^\frac{n-1}n
\leq C_{MS}\left(\int_{\p\S}\varphi\rd\mcH^{n-1}+\int_\S\abs{\na \varphi}\rd\mcH^n+\int_\S \varphi\abs{H}\rd\mcH^n\right),
}
where $C_{MS}>0$ is a constant depending only on $n$.
Using the trace estimate \eqref{ineq:trace-estimate}, we obtain
\eq{
\left(\int_\S{\varphi}^\frac{n}{n-1}\rd\mcH^n\right)^\frac{n-1}n
\leq C(n,F)\left(\int_\S\abs{\na \varphi}\rd\mcH^n+\int_\S \varphi\abs{H}\rd\mcH^n\right)
}
for some constant $C(n,F)>0$.
Replacing $\varphi$ by $\varphi^2$, then using 
the Cauchy-Schwarz inequality, we find (with $r>0$ fixed)
\eq{
\left(\int_\S{\varphi}^\frac{2n}{n-1}\rd\mcH^n\right)^\frac{n-1}n
\leq C(n,F)\left(\int_\S \varphi\abs{\na\varphi}\rd\mcH^n+\frac1r\int_\S\varphi^2\rd\mcH^n+r\int_\S\varphi^2\abs{H}^2\rd\mcH^n\right).
}
Note that the last integral can be estimated by the stability inequality \eqref{ineq:stability-ineq} 
as follows:
\eq{
\int_\S\varphi^2\abs{H}^2\rd\mcH^n
\leq n\int_\S\varphi^2\abs{h}^2\rd\mcH^n
\leq C(n,F)\int_\S\abs{\na\varphi}^2\rd\mcH^n.
}
Hence the assertion follows from the Cauchy-Schwarz inequality.

\end{proof}

\subsection{Iteration argument}
For a measurable function $\phi$ on $\S$ and any $k>0$, we put
\eq{\label{defn:norm-f_k,r}
\norm{\phi}_{k,r,x_0}
\coloneqq\left(\fint_{\mcB_r(x_0)}\abs{\phi}^k\right)^\frac1k
=\left(\frac{\int_{\mcB_r(x_0)}\abs{\phi}^k}{\mcH^n(\mcB_r(x_0))}\right)^\frac1k.
}

We emphasize again that, in view of Remark \ref{Rem:tilde-F}, we may assume without loss of generality that the uniformly elliptic integrand $F$ satisfies $m_F\geq1$ (so that $\log W_f\geq0$), otherwise we could simply replace $F$ by $\tilde F$.

\begin{proposition}\label{Prop:Mean-value-ineq}
Let $u$ be a $C^2$-solution to \eqref{defn:AMSE} on $\mbR^n_+$ satisfying the free boundary condition in the anisotropic sense \eqref{defn:anisotropic-free-bdry-Intro}.
Let $x_0\in\overline{\mbR^n_+}$ and suppose that $u(x_0)=0$.
For any $k>0$, there exists a positive constant $\mfC_F=\mfC_F(n,k,F)$, such that for any $r>0$,
\eq{\label{ineq:Mean-value-ineq}
\norm{\log W_f}_{\infty,\frac{r}2,x_0}
\leq\mfC_F\norm{\log W_f}_{k,r,x_0}.
}
\end{proposition}
\begin{proof}
We follow the classical De Giorgi-Nash-Moser iteration argument.

For any $r>s>0$, consider a cut-off function $\xi$ supported on $B^{n+1}_{r+s,+}(x_0)$, with $\xi\equiv1$ on $B^{n+1}_{r+\frac{s}2,+}(x_0)$ 
, and $\abs{\bar D\xi}\leq\frac{4}s$  on $B^{n+1}_{r+s,+}(x_0)\setminus B^{n+1}_{r+\frac{s}2,+}(x_0)$.

By \eqref{eq:laplace-log-v} we have
${\rm div}_\S(F^2(\nu)\na_F\log W_f)\geq0$.
For any $l\geq1$ we compute using integration by parts and \eqref{eq:na-v-mu=0}:
\eq{
0
\geq&-\int_\S(\log W_f)^{2l-1}\xi^2{\rm div}_\S(F^2(\nu)\na_F\log W_f)\\
=&(2l-1)\int(\log W_f)^{2l-2}\xi^2
g\left(\na\log W_f,\na_F\log W_f\right)F^2(\nu)+2\int(\log W_f)^{2l-1}\xi g\left(\na\xi,\na_F\log W_f\right)F^2(\nu)\\
\geq&(2l-1)\int(\log W_f)^{2l-2}\xi^2
g\left(\na\log W_f,\na_F\log W_f\right)F^2(\nu)\\
&-\frac{l}2\int(\log W_f)^{2l-2}\xi^2
g\left(\na\log W_f,\na_F\log W_f\right)F^2(\nu)
-\frac2l\int(\log W_f)^{2l}
g\left(\na\xi,\na_F\xi\right)F^2(\nu)\\
=&\left(\frac32l-1\right)\int(\log W_{ f})^{2l-2}\xi^2
g\left(\na\log W_f,\na_F\log W_f\right)F^2(\nu)
-\frac2l\int(\log W_f)^{2l}
g\left(\na\xi,\na_F\xi\right)F^2(\nu),
}
which yields
\eq{\label{ineq:Ding24-(4.5)}
\lambda_{F}m_{F}^2\int\abs{\na(\log W_f)^l}^2\xi^2
\leq&m_{F}^2\int
g\left(\na(\log W_f)^l,\na_F(\log W_f)^l\right)\xi^2\\
\leq&\int
g\left(\na(\log W_f)^l,\na_F(\log W_f)^l\right)\xi^2F^2(\nu)\\
\leq&\frac{4l}{3l-2}\int(\log W_f)^{2l}
g\left(\na\xi,\na_F\xi\right)F^2(\nu)
\leq4\Lambda_{F}M_{F}^2\int(\log W_f)^{2l}\abs{\bar D\xi}^2,
}
where we have used \eqref{ineq:bar-nabla-D-relation-n+1} to switch from $\abs{\na(\log W_{f})^l}^2$ to 
$g\left(\na(\log W_f)^l,\na_F(\log W_f)^l\right)$
and from 
$g\left(\na\xi,\na_F\xi\right)$ to $\abs{\bar D\xi}^2$.

Letting $\varphi=(\log W_f)^l\xi$ in the ($L^2$-) Sobolev inequality \eqref{ineq:Sobolev}, we deduce (recalling the definition of $\xi$)
\eq{
\left(\int_{\mcB_r(x_0)}(\log W_f)^{2l\frac{n}{n-1}}\right)^\frac{n-1}n
\leq&C(n,F)\left(\frac1r\int_{\mcB_{r+s}(x_0)}(\log W_f)^{2l}+r\int\abs{\na(\log W_f)^l}^2\xi^2+r\int(\log W_f)^{2l}\abs{\na\xi}^2\right)\\
\overset{\eqref{ineq:Ding24-(4.5)}}{\leq}&C(n,F)\left(\frac1r\int_{\mcB_{r+s}}(\log W_f)^{2l}+\frac{r}{s^2}\int_{\mcB_{r+s}}(\log W_f)^{2l}\right).
}
By \eqref{esti:graph-ball},
it then follows that (noting that $s<r<r+s<2r$)
\eq{
\norm{(\log W_f)^{2l}}_{\frac{n}{n-1},r,x_0}
\leq&C(n,F)\left(r^{-n}\int_{\mcB_{r+s}(x_0)}(\log W_f)^{2l}+\frac{r^2}{s^2}r^{-n}\int_{\mcB_{r+s}(x_0)}(\log W_f)^{2l}\right)\\
\leq&\underbrace{C(n,F)}_{\eqqcolon\mfc_F}\frac{r^2}{s^2}\norm{(\log W_f)^{2l}}_{1,r+s,x_0}.
}
In other words,
\eq{\label{ineq:Ding21-(4-6)}
\norm{\log W_f}_{\frac{2nl}{n-1},r,x_0}
\leq\mfc_F^\frac1{2l}r^\frac1ls^{-\frac1l}\norm{\log W_f}_{2l,r+s,x_0}.
}
This is the starting step of the classical De Giorgi-Nash-Moser iteration, see e.g., \cite{GT01,HL11},
and one can now obtain this proposition by a standard argument. 
More precisely, one first shows that
\eq{\label{ineq:Ding21-(4-3)}
\norm{\log W_f}_{\infty,\de r,x_0}
\leq C(n,k,F)(1-\de)^{-\frac{2n}k}\norm{\log W_f}_{k,r,x_0}
}
for any $k\geq2$ and $\de\in(0,1)$.
This is done by replacing $l,s,r$ in \eqref{ineq:Ding21-(4-6)} respectively by
\eq{
l_i
=\frac{k}2\left(\frac{n}{n-1}\right)^i,\quad
s_i=2^{-(1+i)}(1-\de)r,\quad 
r_i=r-\sum_{j=0}^is_j
=\de r+s_i\leq r,\quad\forall i\in\mbN\cup\{-1\}
}
and then iterating.

To conclude the proof, it is thus left to prove \eqref{ineq:Mean-value-ineq} for $k\in(0,2)$.
By \eqref{ineq:Ding21-(4-3)},
\eq{\label{ineq:Ding21-(4-12)}
\norm{\log W_f}_{\infty,\de r,x_0}
\leq C(n,F)(1-\de)^{-n}\norm{\log W_f}_{2,r,x_0}.
}
Replacing $r,\de$ in \eqref{ineq:Ding21-(4-12)} by
\eq{
r_0=\de r,\quad
r_i=\de r+\sum^i_{j=1}2^{-j}(1-\de)r
=r-2^{-i}(1-\de)r,\quad
\de_i=\frac{r_{i-1}}{r_i},\quad\forall i\in\mbN_+,
}
and noticing that for all $i\geq1$, $r_i\in(\frac12(1+\de)r,r)$,  hence by Proposition \ref{Prop:graph-ball-EVG}
\eq{
\mcH^n(\mcB_r(x_0))
\leq C^\star r^n
=C(n,F)(\frac{r}2)^n
\leq C(n,F)\mcH^n(\mcB_{r_i}(x_0)),\quad\forall i\in\mbN_+.
}
Also observe that
\eq{
1-\de_i
=\frac{r_i-r_{i-1}}{r_i}
=\frac{2^{-i}(1-\de)r}{r-2^{-i}(1-\de)r}
\geq2^{-i}(1-\de).
}
Hence,
using \eqref{ineq:Ding21-(4-12)}, for $i\geq1$ we find (note that $\de_ir_i=r_{i-1}$)
\eq{
\norm{\log W_f}_{\infty,r_{i-1},x_0}
\leq&C(n,F)(1-\de_i)^{-n}\norm{\log W_f}_{2,r_i,x_0}
\leq C(n,F)(1-\de_i)^{-n}\norm{\log W_f}^\frac{k}2_{k,r_i,x_0}\norm{\log W_f}^{1-\frac{k}2}_{\infty,r_i,x_0}\\
\leq&C(n,F)2^{in}(1-\de)^{-n}\norm{\log W_f}^{\frac{k}2}_{k,r,x_0}\norm{\log W_f}^{1-\frac{k}2}_{\infty,r_i,x_0}.
}
Iterating and letting $i\ra\infty$, we obtain for any $k\in(0,2)$ and $\de\in(0,1)$
\eq{
\norm{\log W_f}_{\infty,\de r,x_0}
\leq C(n,k,F)(1-\de)^{-\frac{2n}k}\norm{\log W_f}_{k,r,x_0},
}
which completes the proof.
\end{proof}

\section{Proof of the main results}\label{Sec:5}

\begin{proof}[Proof of Theorem \ref{Thm:gradient-estimate}]
We prove the required estimate at any $x_0\in\overline{\mbR^n_+}$.
After translation, we may always assume that  $u(x_0)=0$.
We continue to use the notations in Section \ref{Sec:4} (in particular, thanks to Remark \ref{Rem:tilde-F} we may assume WLOG $m_F\geq1$, so that $\log W_f\geq0$), and recall that the area elements $W,W_f$ are comparable (see \eqref{ineq:comparable}).

By Lemmas \ref{Lem:integral-estimate-2}, \ref{Lem:integral-estimate-3}, and area formula, we obtain (noting that the projection of $\mcB_r(x_0)=\S\cap B^{n+1}_r(x_0)$ onto $\mbR^n_+$ lies in $\left(B_{r,+}(x_0)\cap\{\abs{u}\leq r\}\right)$)
\eq{
\frac{m_{F}\int_{\mcB_r(x_0)}\log W_f}{\mcH^n(B_{r,+}(x_0))}
&\leq\fint_{B_{r,+}(x_0)\cap\left\{\abs{u}\leq r\right\}}(\log W_f)W_f\\
&\leq C(n,F)r^{-n}\int_{B_{2r,+}(x_0)\cap\{u\geq-2r\}}W_f
\leq C(n,F)\left(\mfc_1+\frac{\mfc_2}{2r}\sup_{B_{4r,+}(x_0)}u\right).
}
Now by the mean value inequality (Proposition \ref{Prop:Mean-value-ineq} with $k=1$ therein), in conjunction with \eqref{esti:graph-ball}, we get
\eq{
\log W_f(x_0)
\leq\left(\tilde{\mathscr{C}_1}+\frac{\tilde{\mathscr{C}_2}}r\sup_{B_{r,+}(x_0)}u\right),
}
for some positive constants $\tilde{\mathscr{C}_1}=\tilde{\mathscr{C}_1}(n,F)$, $\tilde{\mathscr{C}_2}=\tilde{\mathscr{C}_2}(n,F)$.
In view of \eqref{ineq:comparable},
the assertion readily follows.
\end{proof}

\begin{proof}[Proof of Theorem \ref{Thm:Liouville}]
Define the function on $\mbR^n_-=\{x\in\mbR^n:x_1<0\}$:
\eq{
\tilde u(x)
\coloneqq-u(-x).
}
We have
\eq{
D\tilde u(x)
=Du(-x),\quad\forall x\in\mbR^n_-,
}
and it follows from \eqref{defn:AMSE}, \eqref{defn:anisotropic-free-bdry-Intro} that the function $\tilde u$ solves
\eq{
{\rm div}(Df(D\tilde u(x)))=0\text{ on }\mbR^n_-,
}
and satisfies the boundary condition
\eq{
\left<Df(D\tilde u(x),-e_1)\right>=0,\quad\text{on }\p\mbR^n_-=\p\mbR^n_+.
}
Moreover, by \eqref{condi:linear-growth-negative-part} we clearly have
\eq{
\tilde u(x)<\beta(1+\abs{x}),\quad\forall x\in\mbR^n_-.
}
Applying Theorem \ref{Thm:gradient-estimate} to $\tilde u$ we find for any fixed $y\in\overline{\mbR^n_+}$
\eq{
\abs{Du(y)}
=\abs{D\tilde u(-y)}
\leq e^{\mathscr{C}_1+\frac{\mathscr{C}_2}r\left(\sup_{B_{r,-}(-y)}\tilde u-\tilde u(-y)\right)}
\leq e^{\mathscr{C}_1+{\mathscr{C}_2\beta}(1+\frac{1+\abs{y}}r)+{\mathscr{C}_2}\frac{u(y)}r},\quad\forall r>0.
}
Letting $r\ra\infty$, this yields for the fixed $y\in\overline{\mbR^n_+}$
\eq{
\abs{Du(y)}
\leq C(n,\beta,F).
}
Namely, we obtain the uniform gradient bound:
\eq{\label{esti:uniform-gradient-bounds}
\sup_{\overline{\mbR^n_+}}\abs{Du}
\leq C(n,\beta,F).
}
{From \eqref{esti:uniform-gradient-bounds} it is more or less standard to show that $u$ must be affine. For the convenience of the reader we provide the argument.}
First, after translation we may suppose that $u(0)=0$.
For any $R>1$, denote the rescaled function $u_R(x)\coloneqq\frac{u(Rx)}{R}, \forall x\in{\mbR^n_+}$.
    We clearly have
    \eq{
    Du_R(x)=Du(Rx),\quad\forall x\in B_{1,+}(0),
    }
    and it follows from \eqref{defn:AMSE}, \eqref{defn:anisotropic-free-bdry-Intro} that $u_R$ solves
    \eq{
    {\rm div}\left(Df(Du_R(x))\right)=0
\quad\text{on }B_{1,+}(0),
    }
    with boundary condition
    \eq{
    \left<Df(Du_R(x)),e_1\right>=0,\quad\text{on }\p B_{1,+}(0)\cap\p\mbR^n_+.
    }
    By \eqref{esti:uniform-gradient-bounds}, $\abs{Du_R}\leq C(n,\beta,F)$ in $\overline{B_{1,+}(0)}$.
    Moreover, by $u(0)=0$ and the gradient bound it is easy to see that
    \eq{
    \abs{u_R}(x)\leq C(n,\beta,F),\quad\forall x\in B_{1,+}(0).
    }

 Thanks to Proposition \ref{fF}, we observe that 
    \begin{itemize}
        \item $\abs{D^2f(Du_R)}_{C^0(B_1)}$ is bounded from above by some positive constant depending only on $n,F$. 
    \end{itemize}
   
    \begin{itemize}
        \item The eigenvalues of $D^2f(Du_R)$ in $B_1$ has positive lower bound depending only on $n,\beta,F$;
        \item The boundary condition \eqref{defn:anisotropic-free-bdry-Intro} is an oblique boundary condition, since 
        \eq{
        e_1D^2f(Du_R(x))e_1^T
        \geq C(n,\beta,F)>0.
        }
    \end{itemize}
    With these three properties,
    by standard H\"older estimates of quasilinear elliptic equations with oblique boundary conditions \cite{LT86}
    (in particular, Theorem 4.1 therein together with its proof), we obtain $$\|u_R\|_{C^{1,\alpha}(B_{\frac12,+}(0))}\le C,$$ where $C$ is a positive constant independent of $R$.
In particular, this yields $$\Abs{D u_R(x)-D u_R(0)}\le C\abs{x}^{\alpha}$$ for any $x\in B_{\frac12,+}(0)$ and thus, for any $y\in B_{\frac{R}2,+}(0)$,
\eq{
\Abs{D u(y)-D u(0)}
\le C\frac{\abs{y}^{\alpha}}{R^{\alpha}}.
}
For any fixed $y$, letting $R\rightarrow \infty$ we obtain $\Abs{D u(y)-D u(0)}=0$, and thus $u$ must be affine.
\end{proof}

\bibliography{BibTemplate.bib}
\bibliographystyle{amsplain}
\end{document}